\documentclass[12pt]{amsart}
\usepackage[letterpaper,margin=1.2in]{geometry}
\usepackage[utf8]{inputenc}
\usepackage{amsthm,amsmath,amsfonts}
\usepackage{amsmath,amsfonts,amsthm,amssymb,mathrsfs,dsfont,bbm} 
\usepackage{mathtools}
\usepackage{amsmath,amsfonts}
\usepackage{bbm}
\usepackage{hyperref}
\usepackage{changepage}
\usepackage{enumerate}
\usepackage{verbatim}
\usepackage{xcolor}
\usepackage{algorithm}
\usepackage{algorithmic}

\newcommand{\R}{\mathbb{R}}

\newtheorem{theorem}{Theorem}
\newtheorem{lemma}{Lemma}

\newtheorem{proposition}{Proposition}
\theoremstyle{definition}

\newtheorem{example}{Example}
\newtheorem{condition}{Condition}
\begin{document}

\title{Recovering a Magnitude-Symmetric Matrix from its Principal Minors}
\author{Victor-Emmanuel Brunel, John Urschel}

\maketitle

\begin{abstract}
  We consider the inverse problem of finding a magnitude-symmetric matrix (matrix with opposing off-diagonal entries equal in magnitude) with a prescribed set of principal minors. This problem is closely related to the theory of recognizing and learning signed determinantal point processes in machine learning, as kernels of these point processes are magnitude-symmetric matrices. In this work, we prove a number of properties regarding sparse and generic magnitude-symmetric matrices. We show that principal minors of order at most $\ell$, for some invariant $\ell$ depending only on principal minors of order at most two, uniquely determines principal minors of all orders. In addition, we produce a polynomial-time algorithm that, given access to principal minors, recovers a matrix with those principal minors using only a quadratic number of queries. Furthermore, when principal minors are known only approximately, we present an algorithm that approximately recovers a matrix, and show that the approximation guarantee of this algorithm cannot be improved in general.
\end{abstract}

\section{Introduction}\label{sec:intro}

In this paper, we study an algebraic problem that is inherently related to the statistical properties of finite determinantal point processes (DPPs). A DPP is a random subset of a finite ground set whose probability distribution is characterized by the principal minors of some fixed matrix. DPPs emerge naturally in many probabilistic settings such as random matrices and integrable systems \cite{borodin2009determinantal}, and have also recently attracted interest in machine learning for their ability to model random choices while being tractable and mathematically elegant \cite{kulesza2012determinantal}. Given a finite set, say, $[N]=\{1,\ldots,N\}$ where $N$ is a positive integer, a DPP is a random subset $Y\subseteq [N]$ such that $P(S\subseteq Y)=\det(K_S)$, for all fixed $S\subseteq [N]$, where $K\in\R^{N\times N}$ is a given matrix called the kernel of the DPP and $K_S=(K_{i,j})_{i,j\in S}$ is the principal submatrix of $K$ associated with the set $S$. Assumptions on $K$ that yield the existence of a DPP can be found in \cite{BrunelSigned} and it is easily seen that uniqueness of the DPP is then automatically guaranteed. For instance, if $I-K$ is invertible ($I$ being the identity matrix), then $K$ is the kernel of a DPP if and only if $K(I-K)^{-1}$ is a $P_0$-matrix, i.e., all its principal minors are nonnegative (see \cite{johnson1995convex} for properties of $P_0$-matrices). When $I-K$ is invertible, the DPP is also an $L$-ensemble, with probability mass function given by $P(Y=S)=\det(L_S)/\det(I+L)$, for all $S\subseteq [N]$, where $L=K(I-K)^{-1}$ \cite[Theorem 5.1]{borodin2009determinantal}. DPPs with a symmetric kernel have attracted a great deal of interest in machine learning because they satisfy a property called negative association, which models repulsive interactions between items \cite{borcea2009negative}. For instance, when applied to recommender systems, DPPs enforce diversity in the items within one basket \cite{gartrell2017low}. Recently, DPPs with non-symmetric kernels have also gained interest in the machine learning community \cite{anari2021simple,BrunelSigned,gartrell2019learning,poulson2020high}.

From a statistical point of view, DPPs raise two essential questions. First, what kernels are equivalent, in the sense that they give rise to one and the same DPP? Second, given observations of independent realisations of a DPP, how to recover its kernel (which, as foreseen by the first question, is not necessarily unique)? These two questions are directly related to the \textit{principal minor assignment problem} (PMA). Given a class of matrices, (1) decide if there exists a matrix with a prescribed list of principal minors, and, if such a matrix does exist, (2) find one such matrix. While the first task is theoretical in nature, the second one is algorithmic and should be solved relatively quickly and using as few queries of the prescribed principal minors as possible. 

(PMA) is well understood for symmetric matrices. Given a symmetric matrix $K_0\in\R^{N\times N}$, the set of symmetric matrices with the same list of principal minors as $K_0$ is exactly given by all the matrices of the form $DK_0D$ for some diagonal matrix $D$ with $1$ and $-1$ on its diagonal \cite{oeding2011set}. Moreover, given access to the list of principal minors of $K_0$, a polynomial-time algorithm that finds a solution is given in \cite{rising2015efficient}, using combinatorial properties of the sparsity graph of $K_0$, which can be deduced from its principal minors. In \cite{urschel2017learning}, by analyzing the cycle space of the sparsity graph, a graph invariant $\ell$ called {\it cycle sparsity} was considered, and it was shown that principal minors of order at most $\ell$ uniquely determine principal minors of all orders. This led to a polynomial-time algorithm that computes a solution using principal minors of minimal length, and an optimal (w.r.t. sample complexity) algorithm for learning the kernel of a symmetric DPP from samples -- i.e., when only a random noisy version of the principal minors is available -- under some mild separability assumptions.

Classes containing non-symmetric matrices have also attracted some interest, yet the problem is significantly harder in general. A matrix $K\in\R^{N\times N}$ is called irreducible if there is no proper subset $S\subset [N]$ such that either $K_{S,\bar S}$ or $K_{\bar S,S}$ is zero, where $\bar S$ is the complement of $S$ and, for all subsets $S,T\subseteq [N]$, $K_{S,T}$ is the submatrix of $K$ obtained by keeping only the rows and the columns whose indices are in $S$ and $T$ respectively. A matrix $K$ is called HL-indecomposable if for all subsets $S\subset [N]$ with $2 \le |S| \le N-2$, the matrix $K_{S,\bar S}$ has rank at least two. 
In \cite[Theorem 1]{loewy1986principal}, Loewy proved the following. Let $K,K'\in\R^{N\times N}$ (or more generally, in $F^{N\times N}$ for an arbitrary field $F$), such that $K$ is irreducible and HL-indecomposable. Then, $K$ and $K'$ have the same principal minors if and only if there exists an invertible diagonal matrix $D$ such that either $K'=DKD^{-1}$ or $K'=DK^\top D^{-1}$. In \cite{griffin2006principal}, the authors consider the class of dense matrices whose principal submatrices all have dense Schur complements (i.e., that only contain nonzero entries) and design an iterative algorithm to find a solution in polynomial time. The case of dense skew-symmetric matrices was studied in \cite{boussairi2015skew}, where the authors improved on the more general results of Hartfiel and Loewy \cite{hartfiel1984matrices,loewy1986principal} for this specific class of matrices. In the context of learning signed DPPs, \cite{BrunelSigned} considers matrices that are symmetric in magnitudes, i.e., $|K_{i,j}|=|K_{j,i}|$, for all $i,j \in [N]$. Such matrices are relevant in machine learning applications because they have much more modeling power than symmetric matrices for DPPs. An essential assumption made in \cite{BrunelSigned} is that $K$ is dense (all off-diagonal entries non-zero), which significantly simplifies the analysis. In this work, we consider matrices that are symmetric in magnitudes, as in \cite{BrunelSigned}, but we do not assume that the matrices are dense. In this more general setting, we must  account for different sparsity structures, which also necessitates a combinatorial approach.

\subsection{Our contributions}

We treat both theoretical and algorithmic questions related to recovering a magnitude-symmetric matrix (a matrix $K$ satisfying $|K_{i,j}|=|K_{j,i}|$ for all $i,j \in [N]$) from its principal minors. First, in Section \ref{sec:prin_minors} we define the set of magnitude-symmetric matrices under consideration (sparse matrices with a mild genericity condition on cycles of length four), formally state the questions of interest, make a number of preliminary observations using principal minors of low order, and show key connections to the theory of edge-wise charged ($\epsilon(e)$ equals $+1$ or $-1$ for all edges $e$) graphs\footnote{Such graphs are typically called ``signed graphs." However, in addition to the signing of edges of a graph, we also consider the signing of non-zero entries of a matrix. To avoid ambiguity, we refer to the former as the ``charge" of an edge, and save the term ``sign" exclusively for entries of a matrix.}. In Section \ref{sec:cyclespace}, given a charged graph, we prove a number of combinatorial properties regarding cycle bases for the span of positive cycles (cycles with an even number of edges signed $-1$). In Section \ref{sec:identify} we show that, for a given magnitude-symmetric matrix $K$, the principal minors of length at most $\ell$, for some graph invariant $\ell$ depending only on principal minors of order one and two, uniquely determine principal minors of all orders, and that this quantity $\ell$ is tight (Theorem \ref{thm:ell}). Finally, in Section \ref{sec:eff_alg} we describe an efficient algorithm that, given the principal minors of a magnitude-symmetric matrix, computes a matrix with those principal minors. This algorithm queries only $O(N^2)$ principal minors, all of a bounded order that depends solely on the sparsity of the matrix (Theorem \ref{thm:algorithm}). In addition, in Subsection \ref{sec:noise_alg} we also consider the question of recovery when principal minors are only known approximately, and show that a modification of the aforementioned efficient algorithm, for sufficiently generic matrices and sufficiently small error, recovers the matrix almost exactly (Theorem \ref{thm:apx_alg}). Recovery from approximate principal minors is directly related to learning signed DPPs, as the principal minors of the kernel of a DPP can be well-approximated by independent samples with high probability (using the method of moments).

\subsection{Notation and Definitions}

For any number $N \in \mathbb{N}$, let $[N]=\{1,2,\ldots,N\}$ and $\R^{N\times N}$ be the set of all $N \times N$ real matrices. In this work, we consider matrices $K \in \mathbb{R}^{N \times N}$ satisfying $|K_{i,j}| =|K_{j,i}|$ for $i,j \in [N]$, which we refer to as magnitude-symmetric matrices. If $K\in\R^{N\times N}$ and $S\subseteq [N]$, $K_S:=(K_{i,j})_{i,j\in S}$, and $\Delta_S(K):=\det K_S$ is the principal minor of $K$ associated with the set $S$ ($\Delta_\emptyset(K)=1$ by convention). When it is clear from the context, we simply write $\Delta_S$ instead of $\Delta_S(K)$, and for sets of order at most four, we replace the set itself by its elements, e.g., write $\Delta_{\{1,2,3\}}$ as $\Delta_{1,2,3}$. Similarly, when considering the set difference between a set $S$ and some small set of elements, we may again replace the set by its elements, e.g., write $\Delta_{S\backslash\{1,2,3\}}$ as $\Delta_{S\backslash 1,2,3}$.

In addition, we recall a number of relevant graph-theoretic definitions. In this work, all graphs $G = (V,E)$ are simple, undirected graphs, i.e., $V$ is a nonempty finite set and $E$ is a collection of subsets of $V$ of order two. The vertex set and the edge set of a graph $G$ are denoted by $V(G)$ and $E(G)$, respectively. Given a graph $G=(V,E)$, a subgraph $H$ of $G$, denoted $H \subseteq G$, is a graph $H$ with $V(H)\subseteq V$ and $E(H)\subseteq \{\{i,j\}:i,j\in V(H), \{i,j\}\in E\}$. If $S\subseteq V$, the subgraph of $G$ induced by $S$, denoted $G[S]$, is the subgraph $H \subseteq G$ with $V(H)=S$ and $E(H)=\{\{i,j\}:i,j\in S, \{i,j\}\in E\}$. An {\it articulation point} or {\it cut vertex} of a graph $G$ is a vertex whose removal disconnects the graph. A graph $G$ is said to be {\it two-connected} if it has at least two vertices and has no articulation points. A maximal two-connected subgraph $H$ of $G$ is called a {\it block} of $G$. If a block $H$ is of order two, we say $H$ is a trivial block (as $H$ contains no cycles). Given a graph $G$, a {\it cycle} $C$ is a subgraph of $G$ in which every vertex of $C$ has even degree (within $C$)\footnote{A cycle $C$ may have several connected components.}. A {\it simple cycle} $C$ is a connected subgraph of $G$ in which every vertex of $C$ has degree two. For simplicity, we may sometimes describe $C$ by a traversal of its vertices along its edges, i.e., $C = i_1 \; i_2 \; ... \; i_k \; i_1$. In this description, we have the choice of both the start vertex and the orientation of the cycle. A simple cycle $C$ of a graph $G$ is not necessarily induced, i.e., it does not necessarily hold that $C=G[V(C)]$, and while the cycle itself has all vertices of degree two, this cycle may contain some number of chords in the original graph $G$, and we denote this set of chords by $\gamma(C):=E(G[V(C)]) \backslash E(C)$.

Given any subgraph $H$ of $G$, we can associate with $H$ an incidence vector $\chi_H \in \textsf{GF}(2)^m$, $m = |E|$, where $\chi_H(e) =1$ if and only if $e \in E(H)$, and $\chi_H(e) =0$ otherwise. Here, $\textsf{GF}(2)$ is the field with two elements. Given two subgraphs $H_1,H_2 \subseteq G$, we define their {\it sum} $H_1 + H_2$ as the graph containing all edges in exactly one of $E(H_1)$ and $E(H_2)$ (i.e., their symmetric difference) and no isolated vertices. This corresponds to the graph resulting from the sum of their incidence vectors: $\chi_{H_1+H_2}=\chi_{H_1}+\chi_{H_2}$. The cycle space of $G$ is given by
$$\mathcal{C}(G) = \{ \chi_{C} : C \text{ is a cycle of }G \}\subseteq \textsf{GF}(2)^m.$$
It is a linear space of dimension $\nu := m - N + \kappa(G)$, where $\kappa(G)$ denotes the number of connected components of $G$. The quantity $\nu$ is commonly referred to as the {\it cyclomatic number} or {\it circuit rank}. The {\it cycle sparsity} $\ell$ of the graph $G$ is the smallest number for which the set of incidence vectors $\chi_C$ of cycles of edge length at most $\ell$ spans $\mathcal{C}(G)$.

In this work, we require the use and analysis of graphs $G = ([N],E)$ endowed with a linear Boolean function $\epsilon$ that maps subgraphs of $G$ to $\pm 1$, i.e., $\epsilon(e) \in \{-1,+1\}$ for all $e \in E(G)$ and
$$\epsilon(H) = \prod_{e \in E(H) } \epsilon(e)$$
for all subgraphs $H$. We note that for all subgraphs $H_1,H_2$ of $G$, $\epsilon(H_1+H_2) = \epsilon(H_1)\epsilon(H_2)$. The graph $G$ combined with a linear Boolean function $\epsilon$ is denoted by $G = ([N],E,\epsilon)$ and referred to as a {\it charged} graph. If $\epsilon(H) = +1$ (resp. $\epsilon(H) = -1$), then we say the subgraph $H$ is positive (resp. negative). For simplicity, we often denote $\epsilon(\{i,j\})$, $\{i,j\} \in E(G)$, by $\epsilon_{i,j}$. Given a magnitude-symmetric matrix $K \in \mathbb{R}^{N \times N}$, we define the {\it charged sparsity graph} of $K$ as $G_K = ([N],E,\epsilon)$, where
$$E := \{ i,j \in [N] \, | \, i \ne j, \, |K_{i,j}| \ne 0 \}, \qquad \epsilon(\{i,j\}) := \text{sgn}(K_{i,j} K_{j,i}),$$
and when clear from context, we simply write $G$ instead of $G_K$. We denote the set of the incidence vectors of positive cycles as
$$\mathcal{C}^+(G) = \{ \chi_{C} : C \text{ is a cycle of } G, \, \epsilon(C) = +1 \},$$
which is a linear subspace of $\mathcal C(G)$. As we will see, when $H$ is a block, the space $\mathcal{C}^+(H)$ is spanned by positive simple cycles (Proposition \ref{prop:simple_basis}). We define the {\it simple cycle sparsity} $\ell_+$ of $\mathcal{C}^+(G)$, for a two-connected graph $G$, to be the smallest number for which the set of incidence vectors $\chi_C$ of positive {\it simple} cycles of edge length at most $\ell_+$ spans $\mathcal{C}^+(G)$. If $G$ is not two-connected, then we define $\ell_+$ of $\mathcal{C}^+(G)$ to be the maximum simple cycle sparsity of $\mathcal{C}^+(H)$ over all blocks $H$ of $G$. If $G$ is a trivial block or if it contains no blocks, then, for simplicity, we set $\ell_+ = 2$ (rather than $\ell+ = 0$, which would follow from the above definition).  
The study of bases of $\mathcal{C}^+(H)$, for blocks $H$, consisting of incidence vectors of simple cycles constitutes the major graph-theoretic subject of this work, and has connections to the recovery of a magnitude-symmetric matrix from its principal minors.

\section{Principal Minors and Magnitude-Symmetric Matrices}\label{sec:prin_minors}

In this section, we formally describe the problems/questions of interest and make a number of basic observations that motivate the remainder of this work. Using only principal minors of order one and two, the quantities $K_{i,i}$, $|K_{i,j}|$, and the charged sparsity graph $G_K$ can be computed immediately, as $K_{i,i} = \Delta_{i}$ for $i \in [N]$, and  
$$|K_{i,j}| = \sqrt{|\Delta_{i} \Delta_{j} - \Delta_{i,j}|} \quad \text{ and } \quad  \epsilon_{i,j}=\text{sgn}(\Delta_{i} \Delta_{j} - \Delta_{i,j}) \text{ for $i\ne j$.}$$
The main focus of this work is to obtain further information on $K$ using principal minors of order greater than two, identify which principal minors uniquely determine $K$ (up to some equivalence class), and use this information to efficiently recover $K$ from its principal minors. To avoid the unintended cancellation of terms in a principal minor, in what follows we assume that $K$ is {\it generic} in the sense that it satisfies the following condition.
\begin{condition} \label{eqn:gen_prop}
If $K_{i,j} K_{j,k} K_{k,\ell} K_{\ell,i} \ne 0 $ for some distinct $i,j,k,\ell \in [N]$, then
\begin{itemize}
    \item[(i)] $|K_{i,j} K_{k,\ell}| \ne |K_{j,k} K_{\ell,i}|$;
    \item[(ii)] For each $\phi_1,\phi_2,\phi_3 \in \{-1,1\}$, 
    $$\phi_1 K_{i,j} K_{j,k} K_{k,\ell} K_{\ell,i} + \phi_2 K_{i,j} K_{j,\ell} K_{\ell,k} K_{k,i}+\phi_3 K_{i,k} K_{k,j} K_{j,\ell} K_{\ell,i} \ne 0.$$
\end{itemize} 
\end{condition}

The first part of Condition \ref{eqn:gen_prop} implies that the three terms (corresponding to the three distinct cycles on four vertices) in the second part are all distinct in magnitude; the second part strengthens this requirement. 
Condition \ref{eqn:gen_prop} can be thought of as a no-cancellation requirement for four-cycles of principal minors of order four. As we will see, this condition is quite important for the recovery of a magnitude-symmetric matrix from its principal minors. Though, the results of this section, slightly modified, hold under slightly weaker conditions than Condition \ref{eqn:gen_prop}, albeit at the cost of simplicity. This condition rules out dense matrices with an extremely high degree of symmetry, as well as the large majority of rank one matrices, but this condition is satisfied for almost all matrices. If $G_K$ is two-connected, Condition \ref{eqn:gen_prop} implies that $K$ is HL-indecomposable, though the converse is not true (consider, for instance, the adjacency matrix of a signed, unweighted four-cycle $H$ with $\epsilon(H) = -1$).

We denote the set of $N \times N$ magnitude-symmetric matrices satisfying Condition \ref{eqn:gen_prop} by $\mathcal{K}_N$, and, when the dimension is clear from context, we often simply write $\mathcal{K}$. We note that if $K \in \mathcal{K}$, then any matrix $K'$ satisfying $|K'_{i,j}| = |K_{i,j}|$, $i,j \in [N]$, is also in $\mathcal{K}$. In this work, we are interested in the following problems/questions.

\begin{flalign}
    & \text{Given $K \in \mathcal{K}$, what is the minimal $\ell$ such that any $K' \in \mathcal{K}$ with}\label{q2} & \tag{PMA1}  \\
    &\text{$\Delta_S(K) = \Delta_S(K')$ for all $|S| \le \ell$ also satisfies $\Delta_S(K) = \Delta_S(K')$} &\nonumber \\
    &\text{for all $S\subseteq [N]$?} &\nonumber
\end{flalign}
\begin{flalign}
         & \text{Given $K \in \mathcal{K}$, what is the set of $K' \in \mathcal{K}$ that satisfy $\Delta_S(K') = \Delta_S(K)$} & \label{q1} \tag{PMA2}\\
    &\text{for all $S\subseteq [N]$?} &\nonumber 
 \end{flalign}
 \begin{flalign}
         & \text{Given principal minors $\{\Delta_S(K)\}_{S \subset [N]}$ of an unknown matrix $K \in \mathcal{K}$,} \label{q3} \tag{PMA3}\\ & \text{find a matrix $K' \in \mathcal{K}$ that satisfies } \Delta_S(K') = \Delta_S(K) 
    \text{ for all $S\subseteq [N]$.} & \nonumber 
 \end{flalign}
\vspace{2mm}

 \eqref{q2} asks for the smallest $\ell$ such that the principal minors of order at most $\ell$ uniquely determine principal minors of all orders. Using properties of $\mathcal{C}^+(G)$ (treated in Section \ref{sec:cyclespace}), in Section \ref{sec:identify} (Theorem \ref{thm:ell}), we show that the answer depends only on $(\Delta_S)_{|S|\le 2}$ and is the simple cycle sparsity $\ell_+$ of $\mathcal{C}^+(G)$. In Section \ref{sec:eff_alg}, we build on this analysis, and produce a polynomial-time algorithm for recovering a matrix in $\mathcal K$ with prescribed principal minors (possibly given up to some error term), hence treating \eqref{q3}. The polynomial-time algorithm (of Subsection \ref{sec:noise_alg}) for recovery given perturbed principal minors has key connections to learning signed DPPs.

 \eqref{q1} asks to which extent we can hope to recover a matrix in $\mathcal K$ from its principal minors. For instance, the transpose operation $K^T$ and the similarity transformation $D K D$, where $D$ is an involutory diagonal matrix (i.e., $\pm 1$'s on the diagonal), both clearly preserve principal minors. In fact, these two operations suitably combined completely define this set. Even though \cite[Theorem 1]{loewy1986principal} already provides an answer to \eqref{q1} in a more general setting (see Section \ref{sec:intro} for details), we give, in Proposition \ref{thm:set}, an alternate proof that is tailored to our framework and reveals insights for the algorithmic aspect of the problem (which we treat in Section \ref{sec:eff_alg}).

To answer \eqref{q2}, we study properties of certain bases of the space $\mathcal{C}^+(G)$, where $G$ is the charged sparsity graph of $K$. We can make a number of observations regarding the connectivity of $G$ and resulting properties of the principal minors of $K$. If $G$ is disconnected, with connected components given by vertex sets $V_1,...,V_k$, then any principal minor $\Delta_S$ satisfies
\begin{equation}\label{eqn:disconn}
    \Delta_S = \prod_{j=1}^k \Delta_{S \cap V_j}.
\end{equation}
In addition, if $G$ has a cut vertex $i$, whose removal results in connected components with vertex sets $V_1,...,V_k$, then any principal minor $\Delta_S$ satisfies
\begin{equation}\label{eqn:biconn}\Delta_S = \sum_{j_1 = 1}^k \Delta_{[\{i\}\cup V_{j_1}]\cap S} \prod_{j_2 \ne j_1} \Delta_{ V_{j_2} \cap S} - (k-1) \Delta_{\{i\} \cap S} \prod_{j =1 }^k \Delta_{ V_{j} \cap S}.
\end{equation}
This implies that the principal minors of $K$ are uniquely determined by principal minors corresponding to subsets of blocks of $G$. Given this property, we focus on matrices $K$ without an articulation point, i.e. $G$ is two-connected. Given results for matrices without an articulation point and Equations (\ref{eqn:disconn}) and (\ref{eqn:biconn}), we can then answer \eqref{q2} in the more general case.

Next, we make an important observation regarding the contribution of certain terms in the Leibniz expansion of a principal minor. Recall that the Leibniz expansion of the determinant is given by
$$\det(K) = \sum_{\sigma \in \mathcal{S}_N} \text{sgn}(\sigma) \prod_{i=1}^N K_{i,\sigma(i)},$$ 
where $\text{sgn}(\sigma)$ is multiplicative over the (disjoint) cycles forming the permutation, and is preserved when taking the inverse of the permutation, or of any cycle therein. Consider now an arbitrary, possibly non-induced, negative ($\epsilon(C)=-1$) simple cycle $C$ (in the graph-theoretic sense) of the sparsity graph $G$, without loss of generality given by $C = 1\; 2 \; ...\; k\; 1$. Consider the sum of all terms in the Leibniz expansion that contains either the cycle $(1 \; 2 \;  ... \; k)$ or its inverse $(k \;  k-1 \; ... \; 1)$  in the associated permutation. Because
$\epsilon(C) = -1$,
$$K_{k,1} \prod_{i=1}^{k-1} K_{i,i+1} + K_{1,k} \prod_{i=2}^{k} K_{i,i-1} = 0,$$
and the sum over all permutations containing the cyclic permutations $(k \;  k-1 \; ... \; 1)$ or $(1 \; 2 \;  ... \; k)$ is zero. Therefore, terms associated with permutations containing negative cycles do not contribute to principal minors.

To illustrate the additional information contained in higher order principal minors $\Delta_S$, $|S|>2$, we first consider principal minors of order three and four. Consider the principal minor $\Delta_{1,2,3}$, given by 
$$\Delta_{1,2,3} = \Delta_1 \Delta_{2,3} + \Delta_2 \Delta_{1,3} + \Delta_3 \Delta_{1,2} - 2 \Delta_1 \Delta_2 \Delta_3 + \big[1 + \epsilon_{1,2} \epsilon_{2,3} \epsilon_{3,1}\big] K_{1,2} K_{2,3} K_{3,1}.$$
If the corresponding graph $G[\{1,2,3\}]$ is not a cycle ($K_{1,2} K_{2,3} K_{3,1} = 0$), or if the cycle is negative ($\epsilon_{1,2} \epsilon_{2,3} \epsilon_{3,1} = -1$), then $\Delta_{1,2,3}$ can be written in terms of principal minors of order at most two, and contains no additional information about $K$. If $G[\{1,2,3\}]$ is a positive cycle, then we can write $K_{1,2} K_{2,3} K_{3,1}$ as a function of principal minors of order at most three,
\begin{equation}\label{eqn:threecycle}
K_{1,2} K_{2,3} K_{3,1} = \Delta_1 \Delta_2 \Delta_3 - \frac{1}{2} \big[ \Delta_1 \Delta_{2,3} + \Delta_2 \Delta_{1,3}  + \Delta_3 \Delta_{1,2}  \big] + \frac{1}{2} \Delta_{1,2,3},
\end{equation}
which allows us to compute $\text{sgn}(K_{1,2} K_{2,3} K_{3,1})$. This same procedure holds for a positive induced simple cycle of any order. In particular, for a positive induced cycle $C = 1\; 2 \; ...\; k\; 1$, we can compute $\text{sgn}\big(K_{k,1} \prod_{i = 1}^{k-1} K_{i,i+1}\big)$ efficiently (in $O(k^3)$ time) using only $\Delta_{[k]}$ and principal minors of order one and two (see \cite[Equation 1 \& Subsection 2.3]{urschel2017learning} for details). However, when a simple cycle is not induced, further analysis is required. To illustrate some of the potential issues for a non-induced positive simple cycle, we consider principal minors of order four.

Consider the principal minor $\Delta_{1,2,3,4}$. By our previous analysis, all terms in the Leibniz expansion of $\Delta_{1,2,3,4}$ corresponding to permutations with cycles of length at most three can be written in terms of principal minors of size at most three. What remains is to consider the sum of the terms in the Leibniz expansion corresponding to permutations containing a cycle of length four (which there are three pairs of, each pair corresponding to the two orientations of a cycle of length four in the graph sense), which we denote by
\begin{align}\label{eqn:z_4cycle}
    Z &= -K_{1,2}K_{2,3} K_{3,4}K_{4,1} - K_{1,2} K_{2,4} K_{4,3} K_{3,1} - K_{1,3} K_{3,2} K_{2,4} K_{4,1} \\
    &\quad - K_{4,3} K_{3,2} K_{2,1} K_{1,4}  - K_{4,2} K_{2,1} K_{1,3} K_{3,4}  - K_{4,2} K_{2,3} K_{3,1} K_{1,4}.\nonumber
\end{align}
If $G[\{1,2,3,4\}]$ does not contain a positive simple cycle of length four, then $Z = 0$ and $\Delta_{1,2,3,4}$ can be written in terms of principal minors of order at most three. If $G[\{1,2,3,4\}]$ has exactly one positive cycle of length four, without loss of generality given by $C:= 1 \; 2 \; 3 \; 4 \; 1$, then 
\begin{equation}\label{eqn:z4_solo}
    Z = -\big[1 + \epsilon(C) \big] K_{1,2} K_{2,3} K_{3,4} K_{4,1} = - 2 K_{1,2} K_{2,3} K_{3,4} K_{4,1},
\end{equation}
and we can write $K_{1,2} K_{2,3} K_{3,4} K_{4,1}$ as a function of principal minors of order at most four, which allows us to compute $\text{sgn}(K_{1,2} K_{2,3} K_{3,4} K_{4,1})$. If there is more than one positive simple cycle of length four, this implies that all simple cycles of length four are positive and $Z$ is given by
\begin{equation}\label{eqn:z4}
    Z =  -2 \big[ K_{1,2}K_{2,3} K_{3,4}K_{4,1} + K_{1,2} K_{2,4} K_{4,3} K_{3,1} + K_{1,3} K_{3,2} K_{2,4} K_{4,1} \big] .\\
\end{equation}
By Condition \ref{eqn:gen_prop}, the magnitude of each of these three terms is distinct, $Z \ne 0$, and we can compute the sign of each of these terms, as there is a unique choice of $\phi_1,\phi_2,\phi_3 \in \{-1,+1\}$ satisfying 
\begin{equation}\label{eqn:z4_phi} \phi_1 |K_{1,2}K_{2,3} K_{3,4}K_{4,1}| + \phi_2 |K_{1,2} K_{2,4} K_{4,3} K_{3,1}| +\phi_3 |K_{1,3} K_{3,2} K_{2,4} K_{4,1}| = - \frac{Z}{2}.
\end{equation}

This completes the analysis of principal minors of order at most four. In order to better understand the behavior of principal minors of order greater than four, we study a number of properties of $\mathcal{C}^+(G)$ in the following section.

\section{Properties of Minimal Cycle Bases of $\mathcal{C}^+(G)$}\label{sec:cyclespace}

In this section we consider $\mathcal{C}^+(G)$, the set of incidence vectors corresponding to positive cycles, and aim to prove a number of results regarding bases consisting of incidence vectors corresponding to simple cycles. The study of various properties of cycle bases of graphs is of broad interest, with wide-ranging applications, and the work in this section, though necessary for our application, is also of independent interest. We refer the reader to \cite{kavitha2009cycle} for a fantastic survey of the subject. In the following two propositions, we compute the dimension of $\mathcal{C}^+(G)$ and note that, when $G$ is two-connected, this space is indeed spanned by incidence vectors corresponding to positive simple cycles.

\begin{proposition}\label{prop:dim}
Let $G = ([N],E,\epsilon)$. Then $\dim(\mathcal{C}^+(G)) = \nu -1$ if $G$ contains at least one negative cycle, and equals $\nu$ otherwise.
\end{proposition}

\begin{proof}
Consider a basis $\{x_1,...,x_\nu\}$ for $\mathcal{C}(G)$. If all associated cycles are positive, then $\dim(\mathcal{C}^+(G)) = \nu$ and we are done. If $G$ has a negative cycle, then, by the linearity of $\epsilon$, at least one element of $\{x_1,...,x_\nu\}$ must be the incidence vector of a negative cycle. Without loss of generality, suppose that $x_1,...,x_i$ are incidence vectors of negative cycles. Then $x_1 + x_2,..., x_1 + x_i, x_{i+1},...,x_{\nu}$ is a linearly independent set of incidence vectors of positive cycles. Therefore, $\dim(\mathcal{C}^+(G)) \ge \nu-1$. However, $x_1$ is the incidence vector of a negative cycle, and so cannot be in $\mathcal{C}^+(G)$.
\end{proof}

\begin{proposition}\label{prop:simple_basis}
Let $G = ([N],E,\epsilon)$ be two-connected. Then $$\mathcal{C}^+(G) = \text{span} \{ \chi_C \, | \, C \text{ is a simple cycle of }G, \epsilon(C) = +1 \}.$$
\end{proposition}

\begin{proof}
If $G$ does not contain a negative cycle, then the result follows immediately, as then $\mathcal{C}^+(G) = \mathcal{C}(G)$. Suppose then that $G$ has at least one negative cycle. Decomposing this negative cycle into the union of edge-disjoint simple cycles, we note that $G$ must also have at least one negative simple cycle.

Since $G$ is a two-connected graph, it admits a proper (also called open) ear decomposition with $\nu -1$ proper ears (Whitney \cite{Whitney1932}, see also \cite{KorteVygen}) starting from any simple cycle. We choose our initial simple cycle to be a negative simple cycle denoted by $G_0$. Whitney's proper ear decomposition states that we can obtain a sequence of graphs $G_0$, $G_1, \cdots, G_{\nu-1}=G$ where $G_{i}$ is obtained from $G_{i-1}$ by adding a path $P_i$ between two distinct vertices $u_i$ and $v_i$ of $V(G_{i-1})$ with its internal vertices not belonging to $V(G_{i-1})$ (a proper or open ear). By construction, $P_i$ is also a path between $u_i$ and $v_i$ in $G_i$. 

We will prove a stronger statement by induction on $i$, namely that for suitably constructed positive simple cycles $C_j$, $j = 1,..., \nu-1$, we have that, for every $i = 0,1,...,\nu - 1$: \begin{enumerate}
    \item[(i)] $\mathcal{C}^+(G_i)=\text{span} \{ \chi_{C_j}: 1\leq j \leq i\}$,
    \item[(ii)] for every pair of distinct vertices $u,v \in V(G_i)$, there exists both a positive path in $G_i$ between $u$ and $v$ and a negative path in $G_i$ between $u$ and $v$.
\end{enumerate}  

For $i=0$, (i) is clear and (ii) follows from the fact that the two paths between $u$ and $v$ in the cycle $G_0$ must have opposite charge since their sum is $G_0$ and $\epsilon(G_0)=-1$. For $i>0$, we assume that we have constructed $C_1, \cdots, C_{i-1}$ satisfying (i) and (ii) for values less than $i$. To construct $C_i$, we take the path $P_i$ between $u_i$ and $v_i$ and complete it with a path of charge $\epsilon(P_i)$ between $u_i$ and $v_i$ in $G_{i-1}$. The existence of this latter path follows from (ii), and these two paths together form a simple cycle $C_i$ with $\epsilon(C_i)=+1$. It is clear that this $C_i$ is linearly independent from all the previously constructed cycles $C_j$ since $C_i$ contains $P_i$ but $P_i$ was not even part of $G_{i-1}$. This implies (i) for $i$. 

To show (ii) for  $G_i$, we need to consider three cases for $u\neq v$. If $u,v\in V(G_{i-1})$, we can use the corresponding positive and negative paths in $G_{i-1}$ (whose existence follows from (ii) for $i-1$). If $u,v\in V(P_i)$, one of the paths can be the subpath $P_{uv}$ of $P_i$ between $u$ and $v$ and the other path can be $P_i\setminus P_{uv}$ together with a path in $G_{i-1}$ between $u_i$ and $v_i$ of charge equal to $-\epsilon(P_{i})$ (so that together they form a negative cycle). Otherwise, we must have $u\in V(P_i)\setminus \{u_i,v_i\}$ and $v\in V(G_{i-1})\setminus\{u_i,v_i\}$ (or vice versa), and we can select the paths to be the path in $P_i$ from $u$ to $u_i$ together with two oppositely charged paths in $G_{i-1}$ between $u_i$ and $v$. In all these cases, we have shown property (ii) holds for $G_i$. 
\end{proof}

For the remainder of the analysis in this section, we generally assume that $G$ contains a negative cycle, otherwise $\mathcal{C}^+(G)=\mathcal{C}(G)$ and we inherit all of the desirable properties of $\mathcal{C}(G)$. Next, we study the properties of simple cycle bases (bases consisting of incidence vectors of simple cycles) of $\mathcal{C}^+(G)$ that are minimal in some sense. We say that a simple cycle basis $\{\chi_{C_1},...,\chi_{C_{\nu-1}}\}$ for $\mathcal{C}^+(G)$ is lexicographically minimal if it lexicographically minimizes $(|C_1|,|C_2|,...,|C_{\nu-1}|)$, i.e., minimizes $|C_1|$, minimizes $|C_2|$ conditional on $|C_1|$ being minimal, etc. It is easy to show that such a lexicographically minimal basis also has the smallest value of $\sum_{i=1}^\nu |C_i|$, as well as of $\max_{1\leq i\leq \nu} |C_i|$, among all possible bases of $\mathcal C^+(G)$: This follows from the optimality of the greedy algorithm for linear matroids. For brevity, we will simply refer to such a basis as ``minimal." One complicating issue is that, while a minimal cycle basis for $\mathcal{C}(G)$ always consists of induced simple cycles, this is no longer the case for $\mathcal{C}^+(G)$. This for example can happen for an appropriately constructed graph with only two negative, short simple cycles, far away from each other; a minimal (not necessarily simple) cycle basis for $\mathcal{C}^+(G)$ then contains a cycle consisting of the union of these two negative simple cycles. Computing a minimal cycle basis for $\mathcal{C}(G)$ is a well-studied problem, with a number of efficient algorithms to do so (see \cite{kavitha2009cycle} for details). In contrast, the complexity of computing a minimal simple cycle basis for $\mathcal{C}^+(G)$ is not known. Computing the shortest positive simple cycle can be reduced to the computation of the shortest even cycle in a graph, which can be done efficiently (see \cite[Subsection 29.11e]{schrijver2003combinatorial} for details). The complexity of extending this cycle to a basis is not clear.

However, we can make a number of statements regarding chords of simple cycles corresponding to elements of a minimal simple cycle basis of $\mathcal{C}^+(G)$. In the following lemma, we show that a minimal simple cycle basis satisfies a number of desirable properties. Before we do so, we introduce some useful notation. Given a simple cycle $C$, it is convenient to fix an orientation, say, $C = i_1 \; i_2 \; ... \; i_k \; i_1$. Given an orientation and any chord $\{i_{k_1},i_{k_2}\} \in \gamma(C)$, $k_1<k_2$, we denote the two cycles created by this chord by
$$C(k_1,k_2)= i_{k_1} \; i_{k_1 + 1} \; ... \; i_{k_2} \; i_{k_1} \quad \text{and} \quad C(k_2,k_1) = i_{k_2} \; i_{k_2+1} \; ... \; i_{k} \; i_1 \; ... \; i_{k_1} \; i_{k_2}.$$
We have the following result.

\begin{lemma}\label{lm:min_basis}
Let $G = ([N],E,\epsilon)$ be two-connected and $C:= i_1 \; i_2 \;  ... \; i_{k} \; i_1$ be a cycle corresponding to an incidence vector of an element of a minimal simple cycle basis $\{x_{1},...,x_{\nu-1}\}$ of $\mathcal{C}^+(G)$. Then
\begin{enumerate}[(i)]
\item $\epsilon\big(C(k_1,k_2)\big) = \epsilon \big(C(k_2,k_1)\big) = -1$ for all $\{i_{k_1},i_{k_2} \} \in \gamma(C)$,\label{it:min1}
\item if $\{i_{k_1}, i_{k_2}\},\{i_{k_3}, i_{k_4}\} \in \gamma(C)$ satisfy $k_1 < k_3 < k_2 < k_4$ ({\it crossing} chords), then either $k_3 - k_1 = k_4 - k_2 = 1$ or $k_1 = 1$, $k_2 - k_3 = 1$, $k_4 = k$, i.e. these two chords form a four-cycle with two edges of the cycle,\label{it:min2}
\item there does not exist $\{i_{k_1}, i_{k_2}\},\{i_{k_3}, i_{k_4}\}, \{i_{k_5}, i_{k_6}\} \in \gamma(C)$ satisfying either $k_1 < k_2 \le  k_3 < k_4 \le k_5 < k_6$ or $k_6 \le k_1 < k_2 \le k_3 < k_4 \le  k_5$.\label{it:min3}
\end{enumerate}
\end{lemma}

\begin{proof} 
Without loss of generality, suppose that $C = 1 \; 2 \; ... \; k \; 1$. Given a chord $\{k_1,k_2\}$, $C(k_1,k_2) + C(k_2,k_1) = C$, and so $\epsilon(C) = +1$ implies that $\epsilon\big(C(k_1,k_2)\big) = \epsilon\big(C(k_2,k_1)\big)$. If both these cycles were positive, then this would contradict the minimality of the basis, as $|C(k_1,k_2)|,|C(k_2,k_1)|< k$. This completes the proof of Property (\ref{it:min1}).

Given two chords $\{k_1,k_2\}$ and $\{k_3,k_4\}$ satisfying $k_1 < k_3 < k_2 < k_4$, we consider the cycles
\begin{align*}
    C_1 &:= C(k_1,k_2) + C(k_3,k_4), \\
    C_2 &:= C(k_1,k_2) + C(k_4,k_3).
\end{align*}
 By Property (i), $\epsilon(C_1) = \epsilon(C_2)= +1$. In addition, $C_1 + C_2 = C$, and $|C_1|,|C_2|\leq |C|$. By the minimality of the basis, either $|C_1| = |C|$ or $|C_2| = |C|$, which implies that either $|C_1| = 4$ or $|C_2| = 4$ (or both). This completes the proof of Property (\ref{it:min2}).

Given three non-crossing chords  $\{k_1, k_2\}$, $\{k_3, k_4\}$, and $\{k_5,k_6\}$, with $k_1 < k_2 \le k_3 < k_4 \le k_5 < k_6$ (the other case is identical, up to rotation), we consider the cycles \begin{align*}
    C_1 &:= C(k_3,k_4) + C(k_5,k_6) + C, \\
    C_2 &:= C(k_1,k_2) + C(k_5,k_6) + C, \\
    C_3 &:= C(k_1,k_2) + C(k_3,k_4) + C.
\end{align*} 
By Property (i), $\epsilon(C_1) = \epsilon(C_2) = \epsilon(C_3) = +1$. In addition, $C_1 + C_2 + C_3 = C$ and $|C_1|,|C_2|,|C_3|<|C|$, a contradiction to the minimality of the basis. This completes the proof of Property (\ref{it:min3}).
\end{proof}

The proof of the above lemma is algorithmic in nature; in fact, given any simple cycle basis of $\mathcal{C}^+(G)$, one can efficiently derive a basis that satisfies Properties (\ref{it:min1})-(\ref{it:min3}). We make use of this fact in the algorithms of Section \ref{sec:eff_alg} and Appendix \ref{sec:app}. Lemma \ref{lm:min_basis} tells us that in a minimal simple cycle basis for $\mathcal{C}^+(G)$, in each cycle two crossing chords always form a positive cycle of length four (and thus any chord has at most one other crossing chord), and there doesn't exist three non-crossing chords without one chord in between the other two. As we will see in the following lemma, any cycle that satisfies these three properties has quite a rigid structure.

\begin{lemma}\label{lm:span_edge}
Let $C$ be a positive simple cycle of $G=([N],E,\epsilon)$ with $V(C)=\{i_1,i_2,\cdots,i_k\}$ whose chords $\gamma(C)$ satisfy Properties (\ref{it:min1})-(\ref{it:min3}) of Lemma \ref{lm:min_basis}. Then there exist two edges of $C:= i_1 \; i_2 \; ... \; i_k \; i_1$, say, $\{i_1,i_k\}$ and $\{i_\ell, i_{\ell+1}\}$, such that
\begin{enumerate}[(i)]
\item all chords $\{i_p,i_q\} \in \gamma(C)$, $p<q$, satisfy $ p \le \ell < q$, \label{it:span1}
\item any positive simple cycle in $G[V(C)]$ containing $\{i_1,i_k\}$ spans $V(C)$ and contains only edges in $E(C)$ and pairs of crossed chords.\label{it:span2}
\end{enumerate}
\end{lemma}

\begin{proof}
We first note that Property (\ref{it:min1}) of Lemma \ref{lm:min_basis} implies that the charge of a cycle $C'\subseteq G[V(C)]$ corresponds to the parity of $E(C') \cap \gamma(C)$, i.e., if $C'$ contains an even number of chords then $\epsilon(C') = +1$, otherwise $\epsilon(C') = -1$. This follows from constructing a cycle basis for $\mathcal{C}\big(G[V(C)]\big)$ consisting of incidence vectors corresponding to $C$ and $C(i,j)$ for every chord $\{i,j\} \in \gamma(C)$.

Let the cycle $C(i,j)$ be a shortest cycle in $G[V(C)]$ containing exactly one chord of $C$. If $\{i,j\}$ is a non-crossed chord, then we take an arbitrary edge in $E(C(i,j)) \backslash \{i,j\}$ to be $\{i_1,i_k\}$. If this chord crosses another, denoted by $\{i',j'\}$, then either $C(i',j')$ or $C(j',i')$ is also a shortest cycle (possibly both are). If  $C(i',j')$ and $C(j',i')$ are both shortest cycles, then $\gamma(C)$ consists only of the pair of crossing chords $\{i,j\}$ and $\{i',j'\}$, and the result follows immediately. If only one is a shortest cycle, w.l.o.g. $C(i',j')$, then we take an arbitrary edge in $E(C(i,j))\cap E(C(i',j'))$. Because of Property (\ref{it:min3}) of Lemma \ref{lm:min_basis}, there exists $\ell$ such that all chords $\{i_p,i_q\}\in \gamma(C)$, $p<q$, satisfy $p\leq \ell<q$. This completes the proof of Property (\ref{it:span1}).


Next, consider an arbitrary positive simple cycle $C'$ in $G[V(C)]$ containing $\{i_1,i_k\}$. We aim to show that this cycle contains only edges in $E(G)$ and pairs of crossed chords, from which the equality $V(C') = V(C)$ immediately follows. Since $C'$ is positive, it contains an even number of chords $\gamma(C)$, and either all the chords $\gamma(C) \cap E(C')$ are pairs of crossed chords or there exist two chords $\{i_{p_1},i_{q_1}\},\{i_{p_2},i_{q_2}\} \in \gamma(C) \cap E(C')$, w.l.o.g. $p_1 \le p_2 \le \ell < q_2 < q_1$ (we can simply reverse the orientation if $q_1=q_2$), that do not cross any other chord in $\gamma(C) \cap E(C')$. In this case, there would be a $i_1-i{q_2}$ path in $C'$ neither containing $i_{p_1}$ nor $i_{q_1}$ (as $\{i_{p_1},i_{q_1}\}$ would be along the other path between $i_1$ and $i_{q_2}$ in $C'$ and $q_1\neq q_2$). However, this is a contradiction as no such path exists in $C'$, since $\{i_{p_1},i_{q_1}\}$ does not cross any chord in $\gamma(C) \cap E(C')$. Therefore, the chords $\gamma(C) \cap E(C')$ are all pairs of crossed chords. Furthermore, since all pairs of crossed chords must define cycles of length four, we have $V(C')=V(C)$. This completes the proof of Property (\ref{it:span2}).
\end{proof}

Although this is not needed in what follows, observe that, in the above Lemma \ref{lm:span_edge}, $\{i_\ell,i_{\ell+1}\}$ also plays the same role as $\{i_1,i_k\}$ in the sense that any positive simple cycle containing $\{i_\ell,i_{\ell+1}\}$ also spans $V(C)$ and contains only pairs of crossed chords and edges in $E(C)$. In the following section, we make use of the key structural properties (in Lemma  \ref{lm:span_edge}) of minimal cycle bases to answer  \eqref{q2}.

\section{Identifying a Matrix using Principal Minors of Minimal Order}\label{sec:identify}

In this section, using a simple cycle basis of $\mathcal{C}^+(G)$ whose elements satisfy the properties of Lemma \ref{lm:min_basis}, we aim to show that the principal minors of order at most the length of the longest cycle in the basis uniquely determine the principal minors of all orders, thus answering \eqref{q2}. For a cycle $C$, let
$$s(C):=\prod_{\substack{\{i,j\} \in E(C),\\ i<j}}  \text{sgn} \big( K_{i,j} \big).$$
As previously illustrated, when $C$ is a positive simple cycle of length at most four, $s(C)$ can be computed efficiently using only principal minors corresponding to subsets of $V(C)$. In a series of two lemmas, we aim to show that, for an arbitrary cycle $C$ in our basis, we can compute $s(C)$ using $\Delta_{V(C)}$, $(\Delta_S)_{|S|\le 2}$, 
$s(C')$ for simple positive cycles $|C'|\le 4$, and at most eight other principal minors of subsets of $V(C)$. Once we have computed $s(C_i)$ for every simple cycle in our basis, we can compute $s(C)$ for any simple positive cycle by writing $C$ as a sum of some subset of the simple cycles $C_1,...,C_{\nu-1}$ corresponding to incidence vectors in our basis, and taking the product of $s(C_i)$ for all indices $i$ in the aforementioned sum. As we will later see (in Theorem \ref{thm:ell}), this is enough to recover $K$ (up to some equivalence class defined in Proposition \ref{thm:set}). To this end, we prove the following two lemmas.

\begin{lemma}\label{lm:minor_exp}
Let $K \in \mathcal{K}$ have charged sparsity graph $G = ([N],E,\epsilon)$, and $C = i_1 \; ... \; i_k \; i_1$ be a positive simple cycle of $G$ whose chords $\gamma(C)$ satisfy Properties (\ref{it:min1})-(\ref{it:min3}) of Lemma \ref{lm:min_basis}, and with vertices ordered as in Lemma \ref{lm:span_edge}. Then the principal minor corresponding to $S = V(C)$ is given by
\begin{align*}
    \Delta_S &= \Delta_{i_1} \Delta_{S\backslash i_1} + \Delta_{i_k}\Delta_{S\backslash i_k} + \big[ \Delta_{i_1,i_k} - 2\Delta_{i_1} \Delta_{i_k}\big]\Delta_{S\backslash i_1,i_k} \\
     &\quad -2 K_{i_1,i_{k-1}} K_{i_{k-1},i_k} K_{i_k,i_2} K_{i_2,i_1} \Delta_{S \backslash i_1,i_2,i_{k-1},i_k} \\
    &\quad + \big[ \Delta_{i_1,i_{k-1}} - \Delta_{i_1} \Delta_{i_{k-1}}\big]\big[ \Delta_{i_2,i_k} - \Delta_{i_2} \Delta_{i_k}\big]\Delta_{S \backslash i_1,i_2,i_{k-1},i_k} \\
    &\quad + \big[ \Delta_{i_1,i_2} - \Delta_{i_1} \Delta_{i_2}\big]\big[\Delta_{S\backslash i_1,i_2} - \Delta_{i_k} \Delta_{S\backslash i_1,i_2,i_k} \big] \\
    &\quad + \big[ \Delta_{i_{k-1},i_k} - \Delta_{i_{k-1}} \Delta_{i_k}\big]\big[\Delta_{S\backslash i_{k-1},i_k} - \Delta_{i_1} \Delta_{S\backslash i_1,i_{k-1},i_k} \big] \\
    &\quad -\big[ \Delta_{i_1,i_2} - \Delta_{i_1} \Delta_{i_2}\big]\big[ \Delta_{i_{k-1},i_k} - \Delta_{i_{k-1}} \Delta_{i_k}\big]\Delta_{S \backslash i_1,i_2,i_{k-1},i_k} \\
    &\quad + Z.
\end{align*}
where $Z$ is the sum of terms in the Leibniz expansion of $\Delta_{S}$ corresponding to permutations where $\sigma(i_1) = i_{k}$ or  $\sigma(i_{k}) = i_1$, but not both.
\end{lemma}

\begin{proof}
Without loss of generality, suppose that $C = 1 \; 2 \; ... \; k \; 1$. Each term in $\Delta_S$ corresponds to a partition of $S = V(C)$ into a disjoint collection of vertices, pairs corresponding to edges of $G[S]$, and simple cycles $C_i$ of $G[S]$ which can be assumed to be positive. We can decompose the Leibniz expansion of $\Delta_S$ into seven sums of terms, based on how the associated permutation for each term treats the elements $1$ and $k$. Let $X_{j_1,j_2}$, $j_1,j_2 \in \{1,k,*\}$, equal the sum of terms corresponding to permutations where $\sigma(1) = j_1$ and $\sigma(k) = j_2$, where $*$ denotes any element in $\{2,...,k-1\}$. 
The case $\sigma(1) =\sigma(k)$ obviously cannot occur (i.e., $X_{1,1} = X_{k,k} = 0$), and so
$$\Delta_{S} = X_{1,k} + X_{1,*} + X_{*,k} + X_{k,1} + X_{k,*} + X_{*,1} + X_{*,*}.$$
By definition, $Z= X_{k,*} + X_{*,1}$. What remains is to compute the remaining five terms. We have
\begin{align*}
    X_{k,1} &= - K_{1,k} K_{k,1} \Delta_{S\backslash 1,k} = \big[ \Delta_{1,k} - \Delta_{1} \Delta_{k}\big]\Delta_{S\backslash 1,k}, \\
    X_{1,k} &=\Delta_{1} \Delta_{k} \Delta_{S\backslash 1,k}, \\
    X_{1,k} + X_{1,*} &= \Delta_{1} \Delta_{S\backslash 1}, \\
    X_{1,k} + X_{*,k} &= \Delta_{k}\Delta_{S\backslash k},
\end{align*}

and so
$$\Delta_{S} = \Delta_{1} \Delta_{S\backslash 1} + \Delta_{k}\Delta_{S\backslash k} + \big[ \Delta_{1,k} - 2\Delta_{1} \Delta_{k}\big]\Delta_{S\backslash 1,k}+ X_{*,*} + Z.$$

The sum $X_{*,*}$ corresponds to permutations where $\sigma(1) \notin \{1,k\}$ and $\sigma(k) \notin\{1,k\}$. We note that, by Property (\ref{it:min2}) of Lemma \ref{lm:min_basis}, either both $1$ and $k$ contain exactly one chord incident to them, and these two chords are $\{1,k-1\}$ and $\{2,k\}$, or at most one of $1$ and $k$ contains a chord incident to them. We first show that the only permutations that satisfy these two properties take the form $\sigma^2(1) =1$ or $\sigma^2(k) = k$ (or both), or contain both $1$ and $k$ in the same cycle. Suppose to the contrary, that there exists some permutation where $1$ and $k$ are not in the same cycle, and both are in cycles of length at least three. Then in $G[S]$ both vertices $1$ and $k$ contain a chord emanating from them. Therefore, each has exactly one chord and these chords are $\{1,k-1\}$ and $\{2,k\}$. The cycle containing $1$ must also contain $2$ and $k-1$, and the cycle containing $k$ must also contain $2$ and $k-1$, a contradiction. 

In addition, we note that, by the above analysis, the only cycles (of a permutation) that can contain both $1$ and $k$ without having $\sigma(1) = k$ or $\sigma(k) = 1$ are given by $(1 \; 2 \; k \; k-1)$ and $(1 \; k-1 \; k \; 2)$. Furthermore, if $1$ and $k$ are not in the same cycle, then $(1 \; 2)$ and $(k-1 \; k)$, or $(1 \; k-1)$ and $(2 \; k)$ are in the permutation. Therefore, we can decompose $X_{*,*}$ further into the sum of five terms. Let
\begin{itemize}
    \item $Y_1$ equal the sum of terms corresponding to permutations containing $(1 \; 2 \; k \; k-1)$ or $(1 \; k-1 \; k \; 2)$,
    \item $Y_2$ equal the sum of terms corresponding to permutations containing $(1 \; k-1 )$ and $(2 \; k)$,
    \item $Y_3$ equal the sum of terms corresponding to permutations containing $(1 \; 2 )$ and $(k-1 \; k)$,
    \item $Y_4$ equal the sum of terms corresponding to permutations containing $(1 \; 2)$ and not containing $(k-1 \; k)$ or $(k)$,
    \item $Y_5$ equal the sum of terms corresponding to permutations containing $(k-1 \; k)$ and not containing $(1 \; 2)$ or $(1)$.
\end{itemize}
$Y_1$ and $Y_2$ are only non-zero if $\{1,k-1\}$ and $\{2,k\}$ are both chords of $C$, in which case these are the only chords adjacent to $1$ or $k$ and so $Y_4 = Y_5 = 0$. In general, $Y_4$ is only non-zero if there is a chord (other than possibly $\{2,k\}$) incident to $k$, and $Y_5$ is only non-zero if there is a chord (other than possibly $\{1,k-1\}$) incident to $1$. We have $X_{*,*} = Y_1 + Y_2 + Y_3 + Y_4 + Y_5$, and
\begin{align*}
Y_1 &= -2 K_{1,k-1} K_{k-1,k} K_{k,2} K_{2,1} \Delta_{S \backslash 1,2,k-1,k}, \\
Y_2 &= K_{1,k-1} K_{k-1,1} K_{2,k} K_{k,2} \Delta_{S \backslash 1,2,k-1,k} \\&= \big[ \Delta_{1,k-1} - \Delta_{1} \Delta_{k-1}\big]\big[ \Delta_{2,k} - \Delta_{2} \Delta_{k}\big]\Delta_{S \backslash 1,2,k-1,k}, \\
Y_3 &= K_{1,2} K_{2,1} K_{k-1,k} K_{k,k-1} \Delta_{S \backslash 1,2,k-1,k} \\&= \big[ \Delta_{1,2} - \Delta_{1} \Delta_{2}\big]\big[ \Delta_{k-1,k} - \Delta_{k-1} \Delta_{k}\big]\Delta_{S \backslash 1,2,k-1,k}, \\
Y_3 + Y_4 &= -K_{1,2} K_{2,1} \big[\Delta_{S\backslash 1,2} - \Delta_{k} \Delta_{S\backslash 1,2,k} \big] \\&= \big[ \Delta_{1,2} - \Delta_{1} \Delta_{2}\big]\big[\Delta_{S\backslash 1,2} - \Delta_{k} \Delta_{S\backslash 1,2,k} \big], \\
Y_3 + Y_5 &= -K_{k-1,k} K_{k,k-1} \big[\Delta_{S\backslash k-1,k} - \Delta_{1} \Delta_{S\backslash 1,k-1,k} \big] \\&= \big[ \Delta_{k-1,k} - \Delta_{k-1} \Delta_{k}\big]\big[\Delta_{S\backslash k-1,k} - \Delta_{1} \Delta_{S\backslash 1,k-1,k} \big].
\end{align*}
Combining all of these terms gives us
\begin{align*}
    X_{*,*} &= -2 K_{1,k-1} K_{k-1,k} K_{k,2} K_{2,1} \Delta_{S \backslash 1,2,k-1,k} \\
    &\quad + \big[ \Delta_{1,k-1} - \Delta_{1} \Delta_{k-1}\big]\big[ \Delta_{2,k} - \Delta_{2} \Delta_{k}\big]\Delta_{S \backslash 1,2,k-1,k} \\
    &\quad + \big[ \Delta_{1,2} - \Delta_{1} \Delta_{2}\big]\big[\Delta_{S\backslash 1,2} - \Delta_{k} \Delta_{S\backslash 1,2,k} \big] \\
    &\quad + \big[ \Delta_{k-1,k} - \Delta_{k-1} \Delta_{k}\big]\big[\Delta_{S\backslash k-1,k} - \Delta_{1} \Delta_{S\backslash 1,k-1,k} \big] \\
    &\quad -\big[ \Delta_{1,2} - \Delta_{1} \Delta_{2}\big]\big[ \Delta_{k-1,k} - \Delta_{k-1} \Delta_{k}\big]\Delta_{S \backslash 1,2,k-1,k}.
\end{align*}
Combining our formula for $X_{*,*}$ with our formula for $\Delta_S$ leads to the desired result.
\end{proof}

\begin{lemma}\label{lm:z_form}
Let $K \in \mathcal{K}$ have charged sparsity graph $G = ([N],E,\epsilon)$, and $C = i_1 \; ... \; i_k \; i_1$ be a positive simple cycle of $G$ whose chords $\gamma(C)$ satisfy Properties (\ref{it:min1})-(\ref{it:min3}) of Lemma \ref{lm:min_basis}, with vertices ordered as in Lemma \ref{lm:span_edge}. Let $Z$ equal the sum of terms in the Leibniz expansion of $\Delta_{S}$, $S = V(C)$, corresponding to permutations where $\sigma(i_1) = i_{k}$ or  $\sigma(i_{k}) = i_1$, but not both (i.e., as in Lemma \ref{lm:minor_exp}). 
Let $U \subseteq S^2$ be the set of pairs $(a,b)$, $a<b$, for which
$\{i_a,i_{b-1}\},\{i_{a+1},i_b\} \in \gamma(C)$, i.e., cycle edges $\{i_a,i_{a+1}\}$ and $\{i_{b-1},i_b\}$ have a pair of crossed chords between them. Then
$$Z = 2 \,  (-1)^{k+1}  \, K_{i_k,i_1} \prod_{j=1}^{k-1} K_{i_j,i_{j+1}} \prod_{(a,b) \in U} \left[1-\frac{\epsilon_{i_a,i_{a+1}} K_{i_{b-1},i_a} K_{i_{a+1},i_b}}{K_{i_a,i_{a+1}} K_{i_{b-1},i_b}}\right].
$$
\end{lemma}

\begin{proof}
Without loss of generality, suppose that $G$ is labelled so that $C = 1 \; 2 \; ... \; k \; 1$ and edge $\{1,k\}$ satisfies Property (\ref{it:span2}) of Lemma \ref{lm:span_edge}, and so every positive simple cycle of $G[S]$ containing $\{1,k\}$ spans $S$ and contains only edges in $E(C)$ and pairs of crossing chords. Therefore, we may assume without loss of generality that $G[S]$ contains $p$ pairs of crossing chords, and no other chords. If $p = 0$, then $C$ is an induced cycle and the result follows immediately.

 There are $2^p$ Hamiltonian paths in $G[S]$ joining $1$ and $k$, corresponding to the $p$ binary choices of whether to use each pair of crossing chords or not. Let us denote these paths from $1$ to $k$ by $P^{1,k}_\theta$, where $\theta \in \{0,1\}^{p}$, and $\theta_i$ equals one if the $i^{th}$ crossing is used in the path (ordered based on increasing distance to $\{1,k\}$), and zero otherwise. In particular, $P^{1,k}_0$ corresponds to the path $= 1 \; 2 \; ... \; k$. We only consider paths with the orientation $1 \rightarrow k$, as all cycles under consideration are positive, and the opposite orientation has the same sum. 
 Denoting the product of the terms of $K$ corresponding to the path $P^{1,k}_{\theta} = \ell_1 \; \ell_2 \; ... \; \ell_{k}$, where $\ell_1 = 1$ and $\ell_k = k$, by
$$K\big(P^{1, k}_{\theta}\big) = \prod_{j=1}^{k-1} K_{\ell_j,\ell_{j+1}},$$
we have
$$ Z = 2 \, (-1)^{k+1} \, K_{k,1} K\big(P^{1,k}_0\big) \sum_{\theta \in \mathbb{Z}_2^{p}} \frac{K\big(P^{1, k}_{\theta}\big)}{K\big(P^{1,k}_0\big)}.$$
Suppose that the first possible crossing occurs at cycle edges $\{a,a+1\}$ and $\{b-1,b\}$, $a<b$, i.e., $\{a,b-1\}$ and $\{a+1,b\}$ are crossing chords. Then
$$Z = 2 \, (-1)^{k+1} \, K_{k,1} K\big(P^{1,k}_0\big) \sum_{\theta \in \mathbb{Z}_2^{p}} \frac{K\big(P^{a, b}_{\theta}\big)}{K\big(P^{a,b}_0\big)}.$$
We have
$$\sum_{\theta \in \mathbb{Z}_2^{p}} K\big(P^{a, b}_{\theta}\big) = \sum_{\theta_1 = 0}K\big(P^{a,b}_{\theta}\big) + \sum_{\theta_1 = 1}K\big(P^{a, b}_{\theta}\big),$$
and
\begin{align*}
    \sum_{\theta_1 = 0}K\big(P^{a,b}_{\theta}\big) &= K_{a,a+1} K_{b-1,b} \sum_{\theta' \in \mathbb{Z}_2^{p-1}} K\big(P^{a+1,b-1}_{\theta'}\big),
\end{align*}
\begin{align*}
    \sum_{\theta_1 = 1}K\big(P^{a, b}_{\theta}\big) &= K_{a,b-1} K_{a+1,b} \sum_{\theta' \in \mathbb{Z}_2^{p-1}} K\big(P^{b-1, a+1}_{\theta'}\big) \\
    &= K_{a,b-1} K_{a+1,b} \sum_{\theta' \in \mathbb{Z}_2^{p-1}} K\big(P^{a+1, b-1}_{\theta'}\big) \epsilon\big(P^{a+1, b-1}_{\theta'}\big).
\end{align*}
By Property (\ref{it:min1}) of Lemma \ref{lm:min_basis}, $\epsilon\big( C(a,b-1)\big) = -1$, and so
$$\epsilon\big(P^{a+1, b-1}_{\theta'}\big) = -\epsilon_{b-1,a} \epsilon_{a,a+1} \quad \text{and}  \quad K_{a,b-1} \epsilon\big(P^{a+1, b-1}_{\theta'}\big) = - \epsilon_{a,a+1} K_{b-1,a}.$$
This implies that
\begin{align*}
    \sum_{\theta \in \mathbb{Z}_2^{p}} K\big(P^{a, b}_{\theta}\big)
    &=\big(K_{a,a+1} K_{b-1,b} -  \epsilon_{a,a+1} K_{b-1,a} K_{a+1,b}\big) \sum_{\theta' \in \mathbb{Z}_2^{p-1}} K\big(P^{a+1, b-1}_{\theta'}\big),
\end{align*}
and that 
$$ \sum_{\theta \in \mathbb{Z}_2^{p}} \frac{K\big(P^{a, b}_{\theta}\big)}{K\big(P^{a,b}_0\big)} = \left(1-\frac{\epsilon_{a,a+1} K_{b-1,a} K_{a+1,b}}{K_{a,a+1} K_{b-1,b}}\right)\sum_{\theta' \in \mathbb{Z}_2^{p-1}} \frac{K\big(P^{a+1, b-1}_{\theta'}\big)}{K\big(P^{a+1,b-1}_0\big)}.$$
Repeating the above procedure for the remaining $p-1$ crossings completes the proof.
\end{proof}

Equipped with Lemmas \ref{lm:min_basis}, \ref{lm:span_edge}, \ref{lm:minor_exp}, and \ref{lm:z_form}, we can now make a key observation. Suppose that we have a simple cycle basis $\{x_1,...,x_{\nu-1}\}$ for $\mathcal{C}^+(G)$ whose corresponding cycles all satisfy Properties (\ref{it:min1})-(\ref{it:min3}) of Lemma \ref{lm:min_basis}. Of course, a minimal simple cycle basis satisfies this, but in Section \ref{sec:eff_alg} we will consider alternate bases that also satisfy these conditions and may be easier to compute in practice. For cycles $C$ of length at most four, we have already detailed how to compute $s(C)$, and this computation requires only principal minors corresponding to subsets of $V(C)$. When $C$ is of length greater than four, by Lemmas \ref{lm:minor_exp} and \ref{lm:z_form}, we can also compute $s(C)$, using only principal minors corresponding to subsets of $V(C)$, as the quantity $\text{sgn}(K_{i_1,i_{k-1}} K_{i_{k-1},i_k} K_{i_k,i_2} K_{i_2,i_1})$ in Lemma \ref{lm:minor_exp} corresponds to a positive four-cycle, and in Lemma \ref{lm:z_form} the quantity
$$\text{sgn}\bigg(1 - \frac{\epsilon_{i_a,i_{a+1}} K_{i_{b-1},i_a} K_{i_{a+1},i_b} }{K_{i_a,i_{a+1}} K_{i_{b-1},i_b} }\bigg) $$
equals $+1$ if $|K_{i_a,i_{a+1}} K_{i_{b-1},i_b} | > |K_{i_{b-1},i_a} K_{i_{a+1},i_b} |$ and equals
$$- \epsilon_{i_a,i_{a+1}} \epsilon_{i_{b-1},i_b} \, \text{sgn}\big( K_{i_a,i_{a+1}} K_{i_{a+1},i_b} K_{i_b,i_{b-1}} K_{i_{b-1},i_a} \big) $$
if $|K_{i_a,i_{a+1}} K_{i_{b-1},i_b}| < |K_{i_{b-1},i_a} K_{i_{a+1},i_b}|$. By Condition \ref{eqn:gen_prop}, we have $|K_{i_a,i_{a+1}} K_{i_{b-1},i_b}| \ne |K_{i_{b-1},i_a} K_{i_{a+1},i_b}|$. Therefore, given a simple cycle basis $\{x_1,...,x_{\nu-1}\}$ whose corresponding cycles satisfy Properties (\ref{it:min1})-(\ref{it:min3}) of Lemma \ref{lm:min_basis}, we can compute $s(C)$ for every such cycle in the basis using only principal minors of size at most the length of the longest cycle in the basis. Given this fact, we are now prepared to answer \eqref{q2} (in Theorem \ref{thm:ell}) and provide an alternate proof for \eqref{q1} (in Proposition \ref{thm:set}) through the following proposition and theorem.

\begin{proposition}\label{thm:set}
Let $K \in \mathcal{K}$, with charged sparsity graph $G = ([N],E,\epsilon)$. The set of $K' \in \mathcal{K}$ that satisfy $\Delta_S(K) = \Delta_S(K')$ for all $S\subseteq [N]$ is exactly the set generated by $K$ and the operations \vspace{1.5 mm}
\begin{adjustwidth}{5 em}{0pt}
\begin{enumerate}
    \item[$\mathcal{D}_N$-similarity:]$K \rightarrow D K D$, where $D \in \mathbb{R}^{N\times N}$ is an arbitrary involutory diagonal matrix, i.e., $D$ has entries $\pm 1$ on the diagonal, and
    \item[block transpose:]$K\rightarrow K'$, where $K'_{i,j} = K_{j,i}$ for all $i,j \in V(H)$  for some block $H$ of $G$, and $K'_{i,j} = K_{i,j}$ otherwise.
\end{enumerate}
\end{adjustwidth}
\end{proposition}

\begin{proof}
We first verify that the $\mathcal{D}_N$-similarity and block transpose operations both preserve principal minors. The determinant is multiplicative, and so principal minors are immediately preserved under $\mathcal{D}_N$-similarity, as the principal submatrix of $DKD$ corresponding to $S$ is equal to the product of the principal submatrices corresponding to $S$ of the three matrices $D$, $K$, $D$, and so $\Delta_S(DKD) = \Delta_S(D) \Delta_S(K) \Delta_S(D) = \Delta_S(K)$. For the block transpose operation, we note that the transpose leaves the determinant unchanged. By Equation (\ref{eqn:biconn}), the principal minors of a matrix are uniquely determined by the principal minors of the matrices corresponding to the blocks of $G$. As the transpose of any block also leaves the principal minors corresponding to subsets of other blocks unchanged, principal minors are preserved under block transpose.

What remains is to show that if $K'$ satisfies $\Delta_S(K) = \Delta_S(K')$ for all $S\subseteq [N]$, then $K'$ is generated by $K$ and the above two operations. Without loss of generality, we may suppose that the shared charged sparsity graph $G=([N],E,\epsilon)$ of $K$ and $K'$ is two-connected, as the general case follows from this one. We will transform $K'$ to $K$ by first making them agree for entries corresponding to a spanning tree and a negative edge of $G$, and then observing that by Lemmas \ref{lm:minor_exp} and \ref{lm:z_form}, this property implies that all entries agree.

Let $C'$ be an arbitrary negative simple cycle of $G$, $e \in E(C')$ be a negative edge in $C'$, and $T$ be a spanning tree of $G$ containing the edges $E(C')\, \backslash \,e$. By applying $\mathcal{D}_N$-similarity and block transpose to $K'$, we can produce a matrix that agrees with $K$ for all entries corresponding to edges in $E(T) \cup e$. We perform this procedure in three steps, first by making $K'$ agree with $K$ for edges in $E(T)$, then for edge $e$, and then finally fixing any edges in $E(T)$ for which the matrices no longer agree. 

First, we make $K'$ and $K$ agree for edges in $E(T)$. Suppose that $K'_{p,q} = - K_{p,q}$ for some $\{p,q\} \in E(T)$. Let $U \subseteq [N]$ be the set of vertices connected to $p$ in the forest $T \, \backslash \, \{p,q\}$ (the removal of edge $\{p,q\}$ from $T$), and $\hat D$ be the diagonal matrix with $\hat D_{i,i} = +1$ if $i \in U$ and $\hat D_{i,i} = -1$ if $i \not \in U$. The matrix $\hat D K' \hat D$ satisfies
$$\big[\hat D K' \hat D\big]_{p,q} = \hat D_{p,p}  K'_{p,q} \hat D_{q,q} = - K'_{p,q} = K_{p,q},$$
and $\big[\hat D K' \hat D\big]_{i,j} = K'_{i,j}$ for any $i,j$ either both in $U$ or neither in $U$ (and therefore for all edges $\{i,j\} \in E(T)\, \backslash \, \{p,q\}$). Repeating this procedure for every edge $\{p,q\} \in E(T)$ for which $K'_{p,q} = - K_{p,q}$ results in a matrix that agrees with $K$ for all entries corresponding to edges in $E(T)$.  

Next, we make our matrix and $K$ agree for the edge $e$. If our matrix already agrees for edge $e$, then we are done with this part of the proof, and we denote the resulting matrix by $\hat K$. If it does not agree, then, by taking the transpose of this matrix, we have a new matrix that now agrees with $K$ for all edges $E(T) \cup e$, except for negative edges in $E(T)$. By repeating the $\mathcal{D}_N$-similarity operation again on negative edges of $E(T)$, we again obtain a matrix that agrees with $K$ on the edges of $E(T)$, but now also agrees on the edge $e$, as there is an even number of negative edges in the path between the vertices of $e$ in the tree $T$. We now have a matrix that agrees with $K$ for all entries corresponding to edges in $E(T) \cup e$, and we denote this matrix by $\hat K$.

Finally, we aim to show that agreement on the edges $E(T) \cup e$ already implies that $\hat K = K$. Let $\{i,j\}$ be an arbitrary edge not in $E(T) \cup e$, and $\hat C$ be the simple cycle containing edge $\{i,j\}$ and the unique $i-j$ path in the tree $T$. By Lemmas \ref{lm:minor_exp} and \ref{lm:z_form}, the value of $s_K(C)$ can be computed for every cycle corresponding to an incidence vector in a minimal simple cycle basis of $\mathcal{C}^+(G)$ using only principal minors. Then $s_K(\hat C)$ can be computed using principal minors and $s_K(C')$, as the incidence vector for $C'$ combined with a minimal positive simple cycle basis forms a basis for $\mathcal{C}(G)$. By assumption, $\Delta_S(K) = \Delta_S(\hat K)$ for all $S\subseteq [N]$, and, by construction, $s_K(C') = s_{\hat K}(C')$. Therefore, $s_K(\hat C) = s_{\hat K}(\hat C)$ for all $\{i,j\}$ not in $E(T) \cup e$, and so $\hat K = K$. 
\end{proof}

\begin{theorem}\label{thm:ell}
Let $K \in \mathcal{K}$, with charged sparsity graph $G = ([N],E,\epsilon)$. Let $\ell_+ \ge 3$ be the simple cycle sparsity of $\mathcal{C}^+(G)$. Then any matrix $K' \in \mathcal{K}$ with $\Delta_S(K) = \Delta_S(K')$ for all $|S| \le \ell_+$ also satisfies $\Delta_S(K) = \Delta_S(K')$ for all $S\subseteq [N]$, and there exists a matrix $\hat K \in \mathcal{K}$ with $\Delta_S(\hat K) = \Delta_S( K)$ for all $|S| < \ell_+$ but $\Delta_S(\hat K) \ne \Delta_S(K)$ for some $|S|=\ell_+$.
\end{theorem}

\begin{proof}
The first part of theorem follows almost immediately from Lemmas \ref{lm:minor_exp} and \ref{lm:z_form}. It suffices to consider a matrix $K \in \mathcal{K}$ with a two-connected charged sparsity graph $G = ([N],E,\epsilon)$, as the more general case follows from this one. By Lemmas \ref{lm:minor_exp}, and \ref{lm:z_form}, the quantity $s(C)$ is computable for all simple cycles $C$ in a minimal cycle basis for $\mathcal{C}^+(G)$ (and therefore for all positive simple cycles), using only principal minors of size at most the length of the longest cycle in the basis, which in this case is the simple cycle sparsity $\ell_+$ of $\mathcal{C}^+(G)$. The values $s(C)$ for positive simple cycles combined with the magnitude of the entries of $K$ and the charge function $\epsilon$ uniquely determines all principal minors, as each term in the Leibniz expansion of some $\Delta_S(K)$ corresponds to a partitioning of $S$ into a disjoint collection of vertices, pairs corresponding to edges, and oriented versions of positive simple cycles of $E(S)$. This completes the first portion of the proof.

Next, we explicitly construct a matrix $\hat K$ that agrees with $K$ in the first $\ell_+-1$ principal minors ($\ell_+\ge 3$), but disagrees on a principal minor of order $\ell_+$. To do so, we consider a minimal simple cycle basis $\big\{\chi_{C_1},...,\chi_{C_{\nu-1}}\big\}$ of $\mathcal{C}^+(G)$, ordered so that $|C_1|\ge |C_2| \ge ... \ge |C_{\nu-1}|$. By definition, $\ell_+=|C_1|$. Let $\hat K$ be a matrix satisfying
$$\hat K_{i,i} = K_{i,i} \text{ for }i \in [N], \qquad \hat K_{i,j} \hat K_{j,i} = K_{i,j} K_{j,i} \text{ for }i,j \in [N],$$
and
$$s_{\hat K}(C_i) = \begin{cases}  \text{   } s_{K}(C_i) & \text{if } i > 1 \\ - s_{K}(C_i) & \text{if } i = 1 \end{cases}.$$
 The matrices $\hat K$ and $K$ agree on all principal minors of order at most $\ell_+-1$, for if there was a principal minor where they disagreed, this would imply that there is a positive simple cycle of length less than $\ell_+$ whose incidence vector is not in the span of $\{\chi_{C_2},...,\chi_{C_{\nu-1}}\}$, a contradiction. To complete the proof, it suffices to show that $\Delta_{V(C_1)}(\hat K) \ne \Delta_{V(C_1)}(K)$. 
 
 To do so, we consider three different cases, depending on the length of $C_1$. If $C_1$ is a three-cycle, then $C_1$ is an induced cycle and the result follows immediately. If $C_1$ is a four-cycle, $G[V(C_1)]$ may possibly have multiple positive four-cycles. However, in this case the quantity $Z$ from Equation (\ref{eqn:z4}) is distinct for $\hat K$ and $K$, as $K \in \mathcal{K}$. Finally, we consider the general case when $|C_1|>4$. By Lemma \ref{lm:minor_exp}, all terms of $\Delta_{V(C)}(K)$ (and $\Delta_{V(C)}(\hat K)$) except for $Z$ depend only on principal minors of order at most $\ell_+-1$. We denote the quantity $Z$ from Lemma \ref{lm:minor_exp} for $K$ and $\hat K$ by $Z$ and $\hat Z$ respectively. By Lemma \ref{lm:z_form}, the magnitude of $Z$ and $\hat Z$ depend only on principal minors of order at most four, and so $|Z| = |\hat Z|$. In addition, because $K \in \mathcal{K}$, $Z \ne 0$. The quantities $Z$ and $\hat Z$ are equal in sign if and only if $s_{\hat K}(C_1) = s_{K}(C_1)$, and therefore $\Delta_{V(C_1)}(\hat K) \ne \Delta_{V(C_1)}(K)$.
\end{proof}

\section{Efficiently Recovering a Matrix from its Principal Minors}\label{sec:eff_alg}

In Section \ref{sec:identify}, we characterized the set of magnitude-symmetric matrices in $\mathcal{K}$ that share a given set of principal minors, and noted that principal minors of order $ \le \ell_+$ uniquely determine principal minors of all orders, where $\ell_+$ is the simple cycle sparsity of $\mathcal{C}^+(G)$ and is computable using principal minors of order one and two. In this section, we make use of a number of theoretical results of Sections \ref{sec:cyclespace} and \ref{sec:identify} to formally describe a polynomial-time algorithm to produce a magnitude-symmetric matrix $K$ with prescribed principal minors. We provide a high-level description and discussion of this algorithm in Subsection \ref{sub:efficient}, and save the description of a key subroutine for computing a positive simple cycle basis for Subsection \ref{sub:cycle_alg}. In addition, in Subsection \ref{sec:noise_alg}, we consider the more general setting in which principal minors are only approximately known up to some bounded additive error.

\subsection{An Efficient Algorithm}\label{sub:efficient}

Our algorithm proceeds by completing a number of tasks which we describe below. This procedure is very similar in nature to the algorithm implicitly described and used in the proofs of Proposition \ref{thm:set} and Theorem \ref{thm:ell}. The main difference between the procedure alluded to in Section \ref{sec:identify} and our algorithm is the computation of a positive simple cycle basis. Unlike $\mathcal{C}(G)$, there is no known polynomial-time algorithm to compute a minimal simple cycle basis for $\mathcal{C}^+(G)$, and a decision version of this problem may indeed be NP-hard. Our algorithm, which we denote by \textsc{RecoverK}$\big((\Delta_S)_{S\subseteq [N]}\big)$, proceeds in five main steps: \\

\noindent {\bf Step 1:} Compute $K_{i,i}$, $i \in [N]$, $|K_{i,j}|$, $i \ne j$, and $G_K = ([N],E,\epsilon)$. \\

\begin{adjustwidth}{1.5em}{0pt}
We recall that $K_{i,i} = \Delta_i$ and $|K_{i,j}| = |K_{j,i}| = \sqrt{|\Delta_i \Delta_j - \Delta_{i,j}|}$. The edges $\{i,j\} \in E(G)$ correspond to non-zero off-diagonal entries $|K_{i,j}|\ne 0$, and the function $\epsilon$ is given by $\epsilon_{i,j} = \text{sgn}(\Delta_i \Delta_j - \Delta_{i,j})$. \\
\end{adjustwidth}

\noindent {\bf Step 2:} For every block $H$ of $G$, compute a simple cycle basis $\{x_1,...,x_k\}$ of $\mathcal{C}^+(H)$.\\

\begin{adjustwidth}{1.5em}{0pt}
In the proof of Proposition \ref{prop:simple_basis}, we defined an efficient algorithm to compute a simple cycle basis of $\mathcal{C}^+(H)$ for any two-connected graph $H$. This algorithm makes use of the property that every two-connected graph has an open ear decomposition. Unfortunately, this algorithm has no provable guarantees on the length of the longest cycle in the basis. For this reason, we introduce an alternate efficient algorithm in Subsection \ref{sub:cycle_alg} that computes a simple cycle basis of $\mathcal{C}^+(H)$ consisting of cycles all of length at most $3 \phi_H$, where $\phi_H$ is the maximum length of a shortest cycle between any two edges in $H$, i.e.,
$$\phi_H := \max_{e,e' \in E(H)} \; \min_{\substack{\text{simple cycle } C \\ s.t. \, e,e' \in E(C) }} \; |C|$$
with $\phi_H:=2$ if $H$ is acyclic. We define $\phi_G$ to be the maximum of $\phi_H$ over all blocks $H$ of $G$, and $\phi_G:=2$ if $G$ has no blocks. The existence of a simple cycle through any two edges of a non-trivial two-connected graph can be deduced from the existence of two vertex-disjoint paths between newly added midpoints of these two edges.  The simple cycle basis constructed in Subsection \ref{sub:cycle_alg} also maximizes the number of three- and four-cycles it contains. This is a key property which allows us to limit the number of principal minors that we query. \\
\end{adjustwidth}

\noindent {\bf Step 3:} For every block $H$ of $G$, convert $\{x_1,...,x_k\}$ into a simple cycle basis satisfying Properties (\ref{it:min1})-(\ref{it:min3}) of Lemma \ref{lm:min_basis}.\\

\begin{adjustwidth}{1.5em}{0pt}
If there was an efficient algorithm for computing a minimal simple cycle basis for $\mathcal{C}^+(H)$, by Lemma \ref{lm:min_basis}, we would be done (and there would be no need for this step). However, there is currently no known algorithm for this, and a decision version of this problem may be NP-hard. By using the simple cycle basis from Step 2 and iteratively removing chords, we can create a basis that satisfies the same three key properties (in Lemma \ref{lm:min_basis}) that a minimal simple cycle basis does. In addition, the lengths of the cycles in this basis are no larger than those of Step 2, i.e., no procedure in Step 3 ever increases the length of a cycle.

The procedure for this is quite intuitive. Given a cycle $C$ in the basis, we efficiently check that $\gamma(C)$ satisfies Properties (\ref{it:min1})-(\ref{it:min3}) of Lemma \ref{lm:min_basis}. If all properties hold, we are done, and check another cycle in the basis, until all cycles satisfy the desired properties. In each case, if a given property does not hold, then, by the proof of Lemma \ref{lm:min_basis}, we can efficiently compute an alternate cycle $C'$, $|C'|<|C|$, that can replace $C$ in the simple cycle basis for $\mathcal{C}^+(H)$, decreasing the sum of cycle lengths in the basis by at least one. Because of this, the described procedure is a polynomial-time algorithm.\\
\end{adjustwidth}

\noindent {\bf Step 4:} For every block $H$ of $G$, compute $s(C)$ for every cycle $C$ in the basis. \\

\begin{adjustwidth}{1.5em}{0pt}
This calculation relies heavily on the results of Section \ref{sec:identify}. The simple cycle basis for $\mathcal{C}^+(H)$ satisfies Properties (\ref{it:min1})-(\ref{it:min3}) of Lemma \ref{lm:min_basis}, and so we may apply Lemmas \ref{lm:minor_exp} and \ref{lm:z_form}. We compute the quantities $s(C)$ iteratively based on cycle length, beginning with the shortest cycle in the basis and finishing with the longest. By the analysis in Section \ref{sec:prin_minors}, we recall that when $C$ is a three- or four-cycle, $s(C)$ can be computed efficiently using $O(1)$ principal minors, all corresponding to subsets of $V(C)$. When $|C|>4$, by Lemma \ref{lm:minor_exp}, the quantity $Z$ (defined in Lemma \ref{lm:minor_exp}) can be computed efficiently using $O(1)$ principal minors, all corresponding to subsets of $V(C)$. By Lemma \ref{lm:z_form}, the quantity $s(C)$ can be computed efficiently using $Z$, $s(C')$ for positive four-cycles $C' \subseteq G[V(C)]$, and $O(1)$ principal minors all corresponding to subsets of $V(C)$. Because our basis maximizes the number of three- and four-cycles it contains, any such four-cycle $C'$ is either in the basis (and so $s(C')$ has already been computed) or is a linear combination of three- and four-cycles in the basis, in which case $s(C')$ can be computed using Gaussian elimination, without any additional querying of principal minors.  \\
\end{adjustwidth}

\noindent {\bf Step 5:} Output a matrix $K$ satisfying $\Delta_S(K)= \Delta_S$ for all $S \subseteq [N]$. \\

\begin{adjustwidth}{1.5em}{0pt}
The procedure for producing this matrix is quite similar to the proof of Proposition \ref{thm:set}. It suffices to fix the signs of the upper triangular entries of $K$, as the lower triangular entries can be computed using $\epsilon$. For each block $H$, we find a negative simple cycle $C$ (if one exists), fix a negative edge $e \in E(C)$, and extend $E(C)\backslash e$ to a spanning tree $T$ of $H$, i.e., $\big[E(C) \backslash e\big] \subseteq E(T)$. We give the entries $K_{i,j}$, $i<j$, $\{i,j\} \in E(T)\cup e$ an arbitrary sign, and note our choice of $s_K(C)$. We extend the simple cycle basis for $\mathcal{C}^+(H)$ to a basis for $\mathcal{C}(H)$ by adding $C$. On the other hand, if no negative cycle exists, then we simply fix an arbitrary spanning tree $T$, and give the entries $K_{i,j}$, $i<j$, $\{i,j\} \in E(T)$ an arbitrary sign. In both cases, we have a basis for $\mathcal{C}(H)$ consisting of cycles $C_i$ for which we have computed $s(C_i)$. For each edge $\{i,j\} \in E(H)$ corresponding to an entry $K_{i,j}$ for which we have not fixed the sign of, we consider the cycle $C'$ consisting of the edge $\{i,j\}$ and the unique $i-j$ path in $T$. Using Gaussian elimination, we can write $C'$ as a sum of a subset of the cycles in our basis. As noted in Section \ref{sec:identify}, the quantity $s(C')$ is simply the product of the quantities $s(C_i)$ for cycles $C_i$ in this sum. \\
\end{adjustwidth}

 A few comments are in order. Conditional on the analysis of Step 2, the above algorithm runs in time polynomial in $N$. Of course, the set $(\Delta_S)_{S\subseteq [N]}$ is not polynomial in $N$, but rather than take the entire set as input, we assume the existence of some querying operation in which the value of any principal minor can be queried/computed in polynomial time. Combining the analysis of each step, we can give the following guarantee for the \textsc{RecoverK}$\left((\Delta_S)_{S\subseteq [N]}\right)$ algorithm.

 \begin{theorem}\label{thm:algorithm}
 Let $K\in\mathcal K$ and let $\Delta_S=\Delta_S(K)$, for all $S\subseteq [N]$. 
 \begin{itemize}
     \item The algorithm \textsc{RecoverK}$\big((\Delta_S)_{S\subseteq [N]}\big)$ computes a matrix $K'\in \mathcal{K}$ satisfying $\Delta_S(K')=\Delta_S$ for all $S \subseteq [N]$. This algorithm runs in time polynomial in $N$ and queries at most $O(N^2)$ principal minors, all of order at most $3 \, \phi_G$, where $G$ is the sparsity graph of $K$.
     \item In addition, there exists a matrix $\tilde K \in \mathcal{K}$, with $|\tilde K_{i,j}| =|K_{i,j}|$ for all $i,j \in [N]$, such that any algorithm that computes a matrix with principal minors $(\Delta_S(\tilde K))_{S\subset[N]}$ must query a principal minor of order at least $\phi_G$.
 \end{itemize}
 \end{theorem}

The first part of this theorem provides an upper bound for the complexity of solving \eqref{q3}.
The second part asserts that given any $K\in\mathcal K$, one can change the charges of $K$ such that any algorithm that solves \eqref{q3} given the principal minors of the new matrix needs to query a principal minor of high order - namely, of order at least $\phi_G$.

 \begin{proof}
 Conditional on the existence of the algorithm described in Step 2, i.e., an efficient algorithm to compute a simple cycle basis for $\mathcal{C}^+(H)$ consisting of cycles of length at most $3 \phi_H$ that maximizes the number of three- and four-cycles it contains, by the above analysis we already have shown that the \textsc{RecoverK}$\big((\Delta_S)\big)$ algorithm runs in time polynomial in $N$ and queries at most $O(N^2)$ principal minors.
 
 To construct $K'$, we first consider the uncharged sparsity graph $G = ([N],E)$ of $K$, and let $e,e'\in E(G)$ be a pair of edges for which the quantity $\phi_G$ is achieved (if $\phi_G = 2$, we are done). We aim to define an alternate charge function for $G$, show that the simple cycle sparsity of $\mathcal{C}^+(G)$ is at least $\phi_G$, find a matrix $K'$ that has this charged sparsity graph, and then make use of Theorem \ref{thm:ell}. Consider the charge function $\epsilon$ satisfying $\epsilon(e) = \epsilon(e')=-1$, and equal to $+1$ otherwise. Any positive simple cycle basis for the block containing $e,e'$ must have a simple cycle containing both edges $e,e'$. By definition, the length of this cycle is at least $\phi_G$, and so the simple cycle sparsity of $\mathcal{C}^+(G)$ is at least $\phi_G$. Next, let $K'$ be an arbitrary matrix with $|K'_{i,j}| =|K_{i,j}|$, $i,j \in [N]$, and charged sparsity graph $G = ([N],E,\epsilon)$, with $\epsilon$ as defined above. $K \in \mathcal{K}$, and so $K' \in \mathcal{K}$, and therefore, by Theorem \ref{thm:ell}, any algorithm that computes a matrix with principal minors $(\Delta_S(K'))_{S\subset[N]}$ must query a principal minor of order at least $\ell_+ \ge \phi_G$.
 \end{proof}

\subsection{Computing a Simple Cycle Basis with Provable Guarantees}\label{sub:cycle_alg}

Let $H=([N],E,\varepsilon)$ be a two-connected charged graph. We now describe an efficient algorithm to compute a simple cycle basis for $\mathcal{C}^+(H)$ consisting of cycles of length at most $3 \phi_H$. We first compute a minimal cycle basis $\{\chi_{C_1},...,\chi_{C_{\nu}}\}$ of $\mathcal{C}(H)$, where each $C_i$ is an induced simple cycle (as argued previously, any lexicographically minimal basis for $\mathcal{C}(H)$ consists only of induced simple cycles). Without loss of generality, suppose that $C_1$ is the shortest negative cycle in the basis (if no negative cycle exists, we are done). The set of incidence vectors $\chi_{C_1}+\chi_{C_i}$, $\epsilon(C_i) = -1$, combined with vectors $\chi_{C_j}$, $\epsilon(C_j)=+1$, forms a basis for $\mathcal{C}^+(H)$, which we denote by $\mathcal{B}$. We will build a simple cycle basis for $\mathcal{C}^+(H)$ by iteratively choosing incidence vectors of the form $\chi_{C_1}+\chi_{C_i}$, $\epsilon(C_i) = -1$, replacing each of them with an incidence vector $\chi_{\widetilde C_i}$ corresponding to a positive simple cycle $\widetilde C_i$, and iteratively updating $\mathcal{B}$.


Let $C: = C_1 + C_i$ be a positive cycle in our basis $\mathcal{B}$. If $C$ is a simple cycle, we are done, otherwise $C$ is the union of edge-disjoint simple cycles $F_1,..., F_p$ for some $p>1$. If one of these simple cycles is positive and satisfies
$$\chi_{F_j} \not \in \text{span}\big\{\mathcal{B} \backslash \{\chi_{C_1} + \chi_{C_i}\} \big\},$$
 we simply replace $C$ with $F_j$ and we are done. Otherwise, one of $F_1,\ldots,F_p$ must be a negative simple cycle. Indeed, if all $F_j$'s were positive, then $\chi_{F_j} \in \text{span}\big\{\mathcal{B} \backslash \{\chi_{C_1} + \chi_{C_i}\} \big\}$ for all $j=1,\ldots,p$, yielding that $\chi_C=\chi_{F_1}+\ldots+\chi_{F_p}\in \text{span}\big\{\mathcal{B} \backslash \{\chi_{C_1} + \chi_{C_i}\} \big\}$, which is impossible.
Without loss of generality, assume that $F_1$ is a negative cycle. We can construct a set of $p-1$ positive cycles by taking $F_1+F_j$ for all $j=1,\ldots,p$ with $\varepsilon(F_j)=-1$ and keeping all positive $F_j$'s, and at least one of these cycles is not in
$\text{span}\big\{\mathcal{B} \backslash \{\chi_{C_1} + \chi_{C_i}\} \big\}$. We now have a positive cycle not in this span, given by the sum of two {\it edge-disjoint} negative simple cycles $F_1$ and $F_j$. These two cycles satisfy, by construction, $|F_1|+|F_j|\le |C| \le 2 \ell$, where $\ell$ is the cycle sparsity of $\mathcal{C}(H)$. If $|V(F_1)\cap V(F_j)|>1$, then $F_1 \cup F_j$ is two-connected, and we may compute a simple cycle basis for $\mathcal{C}^+(F_1 \cup F_j)$ using the ear decomposition approach of Proposition \ref{prop:simple_basis}. At least one positive simple cycle in this basis satisfies our desired condition, and the length of this positive simple cycle is at most $|E(F_1 \cup F_j)| \le 2 \ell$. If $F_1 \cup F_j$ is not two-connected, then we compute a shortest cycle $\widetilde C$ in $E(H)$ containing both an edge in $E(F_1)$ and $E(F_j)$ (computed using Suurballe's algorithm, see \cite{suurballe1974disjoint} for details). The graph
$$H' = \big( V(F_1) \cup V(F_j) \cup V(\widetilde C),E(F_1) \cup E(F_j) \cup E(\widetilde C)\big)$$
is two-connected, and we can repeat the same procedure as above for $H'$, where, in this case, the resulting cycle is of length at most $|E(H')| \le 2 \ell + \phi_H$. The additional guarantee for maximizing the number of three- and four-cycles can be obtained easily by simply listing all positive simple cycles of length three and four, combining them with the computed basis, and performing a greedy algorithm. What remains is to note that $\ell \le \phi_H$. This follows quickly from the observation that for any vertex $u$ in any simple cycle $C$ of a minimal cycle basis for $\mathcal{C}(H)$, there exists an edge $\{v,w\} \in E(C)$ such that $C$ is the disjoint union of a shortest $u-v$ path, shortest $u-w$ path, and $\{v,w\}$ \cite[Theorem 3]{horton1987polynomial}.

\subsection{An Algorithm for Principal Minors with Noise}\label{sec:noise_alg}
So far, we have exclusively considered the situation in which principal minors are known (or can be queried) exactly. In applications related to this work, this is often not the case. A key application of a large part of this work is signed determinantal point processes, and the algorithmic question of learning the kernel of a signed DPP from some set of samples. Here, for the sake of readability, we focus on a non-probabilistic setting. When principal minors are approximated by random samples from a DPP, the below results can be easily converted to high probability guarantees via a union bound. Given some unknown matrix $K\in\mathcal K$ with all principal minors satisfying $|\Delta_S(K)|\leq 1, S\subseteq [N]$, each $\Delta_S(K)$ can be queried/estimated up to some absolute error term $0<\delta < 1$, i.e., we have access to a collection $(\hat \Delta_S)_{S \subseteq [N]}$, satisfying 
$$ \big| \hat \Delta_S - \Delta_S(K) \big | \le \delta \qquad \text{for all } S\subseteq [N].$$
Our goal is to compute a matrix $K'$ with small error, measured by the quantity
$$\rho(K,K'):= \; \min \big\{| \hat K - K'|_{\infty} \quad \text{s.t.} \quad \hat K \in \mathcal{K},  \; \Delta_S(\hat K) = \Delta_S(K) \text{ for all } S \subseteq [N]\big\},$$
where $|K -K'|_{\infty} := \max_{1\leq i,j\leq N} | K_{i,j} - K'_{i,j}|$. 
We note that because $|\Delta_S(K)|\le 1$, for all $S\subseteq [N]$, we may always assume that $|\hat \Delta_S|\le 1$ as well, for if it is not, then we may safely replace it with $\text{sgn}(\hat \Delta_S)$. In this setting, we require both a separation condition for the entries of $K$ and a stronger version of Condition \ref{eqn:gen_prop}. Let $\alpha,\beta,\gamma\in (0,1]$.

\begin{condition} \label{cond:approx}
For all $i,j\in [N]$ with $i\neq j$, either $|K_{i,j}|=0$ or $|K_{i,j}|\geq \alpha$ and if $K_{i,j} K_{j,k} K_{k,\ell} K_{\ell,i} \ne 0 $ for some distinct $i,j,k,\ell \in [N]$, then
\begin{itemize}
    \item[(i)] $\big||K_{i,j} K_{k,\ell}|-|K_{j,k} K_{\ell,i}|\big|\geq \beta$;
    \item[(ii)] For each $\phi_1,\phi_2,\phi_3 \in \{-1,1\}$, 
    $$\big|\phi_1 \, K_{i,j} K_{j,k} K_{k,\ell} K_{\ell,i} + \phi_2 K_{i,j} K_{j,\ell} K_{\ell,k} K_{k,i}+\phi_3 K_{i,k} K_{k,j} K_{j,\ell} K_{\ell,i}\big| \ge \gamma.$$
\end{itemize} 
\end{condition}

We denote by $\mathcal{K}_N(\alpha,\beta,\gamma)$ the set of matrices $K \in \mathcal{K}_N$ satisfying Condition \ref{cond:approx} and such that $|\Delta_S(K)|\le 1$ for all $S \subseteq [N]$. 
Condition \ref{cond:approx} is a stronger version of Condition \ref{eqn:gen_prop} that accounts for the error in principal minors.

The algorithm \textsc{RecoverK}, slightly modified, performs well for this more general problem. In Appendix \ref{sec:app}, we define a natural modification, called \textsc{RecoverK$^\approx$}. The following theorem illustrates that for $\delta$ sufficiently small (with respect to $\alpha,\beta,\gamma$) our algorithm performs well.

\begin{theorem}\label{thm:apx_alg}
Let $K \in \mathcal{K}_N(\alpha,\beta,\gamma)$ and $\{\hat \Delta_S\}_{S \subseteq [N]}$ satisfy $| \hat \Delta_S - \Delta_S(K) | \le \delta $ for all $S \subseteq [N]$, for some 
$$\delta < \alpha \min\{ \alpha^{3 \phi_G}, \beta^{3 \phi_G/2},\gamma \}/100,$$
where $G$ is the sparsity graph of $K$. The \textsc{RecoverK$^\approx$}$\big(\{\hat \Delta_S\}, \, \delta \big)$ algorithm outputs a matrix $K'$ satisfying 
$$\rho(K,K') \le 3 \delta/2 \alpha.$$
This algorithm runs in time polynomial in $N$, and queries at most $O(N^2)$ principal minors, all of order at most $3 \, \phi_G$.
\end{theorem}

The formal description of the \textsc{RecoverK$^\approx$} algorithm, as well as the proof of Theorem \ref{thm:apx_alg}, adds little insight to the problem itself, and is reserved for Appendix \ref{sec:app}. The following example illustrates why the requirement on $\delta$ in Theorem \ref{thm:apx_alg} cannot be improved in general (up to multiplicative constants in the exponent) for any algorithm. In the below example, $\delta = O(\min(\alpha^{\phi_G},\beta^{\phi_G/2}))$, and there are two possible choices for $K$ for which any output matrix has $\rho(\cdot,K')\ge \alpha$ for at least one of these two choices.

\begin{example}
Let $G = ([N],E,\epsilon)$, where $N>4$ and even,
$$E = \{1,N\} \cup \big\{\{i,i+1\}\big\}_{i=1}^{N-1} \cup \big\{\{i,N-i\},\{i+1,N-i+1\}\big\}_{i=1}^{N/2 - 1},$$
and 
$$ \epsilon(e) = \begin{cases} -1 & \text{for } e = \{1,N\},\{N/2,N/2+1\}, \\ +1 & \text{otherwise}  \end{cases}.$$
This graph can be thought of as an $N$ vertex cycle $C := 1 \; 2 \; ... \; N \; 1$ with $N/2-1$ pairs of crossing chords between edges $\{i,i+1\}$ and $\{N-i,N-i+1\}$, $i = 1,...,N/2-1$. One key property of this graph is that any positive simple cycle containing $\{1,N\}$ or $\{N/2,N/2+1\}$ must contain both and span $[N]$. Consider $K,\hat K \in \mathcal{K}_n(\alpha,\alpha^2,2\alpha^4)$ with charged sparsity graph $G$, with
$$K_{i,j} = \begin{cases} \sqrt{2} \alpha & \text{for } \{i,j\} \in \big\{\{i,N-i\},\{i+1,N-i+1\}\big\}_{i=1}^{N/2 - 1} \\
\alpha & \text{for } i \equiv j-1 \mod N.
\end{cases},$$
i.e., $K$ equals $\sqrt{2} \alpha$ on the crossing chords, and $\alpha$ on the cycle $C$, with the exception of $K_{1,N} = K_{N/2+1,N/2} = - \alpha$, and
$$\hat K_{i,j} = \begin{cases} - K_{i,j} & \text{for } \{i,j\}  = \{1,N\} \\
K_{i,j} & \text{otherwise} \end{cases}.$$
Any positive simple cycle containing $\{1,N\}$ is of length $N$, and so $K$ and $\hat K$ agree on all principal minors of order less than $N$. Using Lemmas \ref{lm:minor_exp} and \ref{lm:z_form}, we note that for $\Delta_{[N]}$, $K$ and $\hat K$ agree everywhere except for the value of $Z$. Because $s_K(C) = - s_{\hat K}(C)$, $Z_K = - Z_{\hat K}$, and
$$\big|\Delta_{[N]}(K) - \Delta_{[N]}(\hat K) \big| = 2 |Z_{K}| = 4 \alpha^N.$$
Taking $\delta = 2 \alpha^N$, and setting 
$$\hat \Delta_S = \begin{cases} \Delta_S(K) & S \ne [N] \\ (\Delta_{[N]}(K) + \Delta_{[N]}(\hat K))/2 & S = [N] \end{cases},$$
any matrix $K'$ output from the input $\{\hat \Delta_S\}$ has either $\rho(K,K')\ge \alpha$ or $\rho(\hat K,K')\ge \alpha$.
\end{example}

\section*{Acknowledgements}

The work of J. Urschel was supported in part by ONR Research Contract N00014-17-1-2177. This material is based upon work supported by the Institute for Advanced Study and the National Science Foundation under Grant No. DMS-1926686. The second author would like to thank Michel Goemans for many interesting conversations on the subject. An earlier draft of this work appeared in the second author's PhD thesis \cite{urschel2021graphs}. The authors are grateful to Louisa Thomas for greatly improving the style of presentation.

{ \small 
	\bibliographystyle{plain}
	\bibliography{main} }
	
	\newpage
	\appendix
	\section{The RecoverK$^\approx$ Algorithm and a Proof of Theorem \ref{thm:apx_alg}}\label{sec:app}

This Appendix consists of two parts: in Subsection A.1 we define the \textsc{RecoverK$^\approx$} algorithm, and in Subsection A.2 we prove Theorem \ref{thm:apx_alg}.

\subsection{The RecoverK$^\approx$ Algorithm}

The algorithm to compute a $K'$ sufficiently close to $K$,  denoted by \textsc{RecoverK$^\approx$} $\big(\{\hat \Delta_S\}; \delta \big)$, proceeds in three main steps:\\

\noindent {\bf Step 1:} Compute $K'_{i,i}$, $i \in [N]$, $|K'_{i,j}|$, $i \ne j$, and $G_{K'} = ([N],E,\epsilon)$. \\

\begin{adjustwidth}{1.5em}{0pt}
We define $K'_{i,i} = \hat \Delta_i$. The difference between the quantities $K_{i,j}K_{j,i} = \Delta_i \Delta_j - \Delta_{i,j}$ and $\hat \Delta_i \hat \Delta_j - \hat \Delta_{i,j}$ are at most
\begin{equation}\label{eqn:offdiag}
    \big| K_{i,j}K_{j,i} - ( \hat \Delta_i \hat \Delta_j - \hat \Delta_{i,j}) \big| \le \big| \hat \Delta_i \big| \big| \Delta_j - \hat \Delta_j\big| + \big| \Delta_j \big| \big| \Delta_i - \hat \Delta_i \big| + \big| \Delta_{i,j} - \hat \Delta_{i,j} \big| \le 3 \delta,
\end{equation}
and so we impose
$$  K'_{i,j} K'_{j,i}  := \begin{cases} \hat \Delta_i \hat \Delta_j - \hat \Delta_{i,j}  & \text{if} \quad |\hat \Delta_i \hat \Delta_j - \hat \Delta_{i,j} | > 3\delta \\ \qquad 0 & \text{otherwise} \end{cases}.$$
The edges $\{i,j\} \in E(G_{K'})$ correspond to non-zero off-diagonal entries $|K'_{i,j}|\ne 0$, and the function $\epsilon$ is given by $\epsilon_{i,j} = \text{sgn}(K'_{i,j} K'_{j,i})$. By construction, $G_{K'}$ is a subgraph of $G_K$, i.e., all edges of $G_{K'}$ are in $G_{K}$ and agree in sign. \\
\end{adjustwidth}

\noindent {\bf Step 2:} For every block $H$ of $G_{K'}$, compute a simple cycle basis $\{x_1,...,x_k\}$ of $\mathcal{C}^+(H)$ satisfying Properties (\ref{it:min1})-(\ref{it:min3}) of Lemma \ref{lm:min_basis}, and define $s_{K'}(C)$ for every cycle $C$ in the basis. \\

\begin{adjustwidth}{1.5em}{0pt}
The computation of the simple cycle basis depends only on the graph, and so is identical to Steps 1 and 2 of the \textsc{RecoverK} algorithm. The simple cycle basis for $\mathcal{C}^+(H)$ satisfies Properties (i)-(iii) of Lemma \ref{lm:min_basis}, and so we can apply Lemmas \ref{lm:minor_exp} and \ref{lm:z_form}. We define the quantities $s_{K'}(C)$, for each cycle $C$ in our basis, iteratively based on cycle length, beginning with the shortest cycle. Each definition is inspired by the corresponding equations for the exact case.  \\

\noindent {\bf Case I:} $|C|=3$. \\

\noindent Let $C = i \, j \, k \, i$, $i<j<k$. Based on Equation (\ref{eqn:threecycle}) for $K_{i,j} K_{j,k} K_{k,i}$, we define
$$s_{K'}(C) := \epsilon_{i,k} \; \text{sgn}\big(\hat \Delta_i \hat \Delta_j \hat \Delta_k - \big[ \hat \Delta_i \hat \Delta_{j,k} + \hat \Delta_j \hat \Delta_{i,k}  + \hat \Delta_k \hat \Delta_{i,j}  \big]/2 +  \hat \Delta_{i,j,k}/2 \big).$$
\\

\noindent {\bf Case II:} $|C|=4$. \\

\noindent Let $C = i \, j \, k \, \ell \, i$, $i<j<k<\ell$. We note that 
\begin{align*}
    \Delta_{i,j,k,\ell} &= \Delta_{i,j} \Delta_{k,\ell} + \Delta_{i,k} \Delta_{j,\ell} + \Delta_{i,\ell} \Delta_{j,k} + \Delta_i \Delta_{j,k,\ell} + \Delta_j \Delta_{i,k,\ell} +\Delta_k \Delta_{i,j,\ell} \\
    &\quad  +\Delta_\ell \Delta_{i,j,k} - 2  \Delta_i \Delta_j \Delta_{k,\ell} - 2  \Delta_i \Delta_k \Delta_{j,\ell} - 2   \Delta_i \Delta_\ell \Delta_{j,k} - 2 \Delta_j \Delta_k \Delta_{i,\ell} \\
    &\quad - 2 \Delta_j \Delta_\ell \Delta_{i,k} - 2 \Delta_k \Delta_\ell \Delta_{i,j} + 6 \Delta_i \Delta_j \Delta_k \Delta_\ell  + Z,
\end{align*}
where $Z$ is the sum of the terms in the Leibniz expansion of $\Delta_{i,j,k,\ell}$ corresponding to four-cycles (defined in Equation (\ref{eqn:z_4cycle})). For this reason, we define
\begin{align*}
    \hat Z &=  \hat \Delta_{i,j,k,\ell} - \hat \Delta_{i,j} \hat \Delta_{k,\ell} - \hat \Delta_{i,k} \hat \Delta_{j,\ell} - \hat \Delta_{i,\ell} \hat \Delta_{j,k}   - \hat \Delta_i \hat \Delta_{j,k,\ell} - \hat \Delta_j \hat \Delta_{i,k,\ell}  \\
    &\quad -\hat \Delta_k \hat \Delta_{i,j,\ell} -\hat \Delta_\ell \hat \Delta_{i,j,k} +2 \hat \Delta_i \hat \Delta_j \hat \Delta_{k,\ell} + 2 \hat \Delta_i \hat \Delta_k \hat \Delta_{j,\ell} + 2 \hat \Delta_i \hat \Delta_\ell \hat \Delta_{j,k} \\
    &\quad +2 \hat \Delta_j \hat \Delta_k \hat \Delta_{i,\ell}  +2 \hat \Delta_j \hat \Delta_\ell \hat \Delta_{i,k} + 2\hat \Delta_k \hat \Delta_\ell \hat \Delta_{i,j} - 6 \hat \Delta_i \hat \Delta_j \hat \Delta_k \hat \Delta_\ell .
\end{align*}
We consider two cases, depending on the number of positive four-cycles in $G_{K'}[\{i,j,k,\ell\}]$. If $C$ is the only positive four-cycle in $G_{K'}[\{i,j,k,\ell\}]$, then, based on Equation (\ref{eqn:z4_solo}), we define 
\begin{align*}
    s_{K'}(C)&= \epsilon_{i,\ell} \; \text{sgn}(- \hat Z ).
\end{align*}
If there is more than one positive four-cycle in $G_{K'}[\{i,j,k,\ell\}]$, then, based on Equations (\ref{eqn:z4}) and (\ref{eqn:z4_phi}), we compute an assignment of values $\phi_1,\phi_2, \phi_3 \in \{-1,+1\}$ that minimizes (not necessarily uniquely) the quantity
$$\bigg| \hat Z /2 + \phi_1 \big| \hat K_{i,j} \hat K_{j,k} \hat K_{k,\ell} \hat K_{\ell,i} \big| + \phi_2 \big|\hat K_{i,j} \hat K_{j,\ell} \hat K_{\ell,k} \hat  K_{k,i} \big| + \phi_3 \big| \hat K_{i,k}  \hat K_{k,j} \hat K_{j,\ell} \hat K_{\ell,i} \big| \bigg|,$$
and define $s_{K'}(C) = \epsilon_{i,\ell} \; \phi_1$ in this case. In addition, if any other cycle in $G[\{i, j ,k ,\ell\}]$ is in our basis and $s_{K'}(\cdot)$ is not yet defined, we now assign this quantity to be consistent with $C$ (i.e., to be $\epsilon_{k,\ell}\epsilon_{i,k}\phi_2$ for $i \; j \; \ell \; k \; i$ and $\epsilon_{j,k} \epsilon_{i,\ell} \phi_3$ for $i \; k \; j \; \ell \; i$).  \\

\noindent {\bf Case III:} $|C|>4$. \\

\noindent Let $C = i_1\, ... \, i_k \, i_1$, $k>4$, with vertices ordered to match the ordering of Lemma \ref{lm:minor_exp}, and define $S = \{i_1,...,i_k\}$. Based on Lemma \ref{lm:minor_exp}, we define $\hat Z$ to equal
\begin{align*}
    \hat Z &= \hat \Delta_S - \hat \Delta_{i_1} \hat \Delta_{S\backslash i_1} - \hat \Delta_{i_k}\hat \Delta_{S\backslash i_k} - \big[ \hat \Delta_{i_1,i_k} - 2\hat \Delta_{i_1} \hat \Delta_{i_k}\big]\hat \Delta_{S\backslash i_1,i_k} \\
     &\quad +2 K'_{i_1,i_{k-1}} K'_{i_{k-1},i_k} K'_{i_k,i_2} K'_{i_2,i_1} \hat \Delta_{S \backslash i_1,i_2,i_{k-1},i_k} \\
    &\quad -\big[ \hat \Delta_{i_1,i_{k-1}} - \hat \Delta_{i_1} \hat \Delta_{i_{k-1}}\big]\big[ \hat \Delta_{i_2,i_k} - \hat \Delta_{i_2} \hat \Delta_{i_k}\big]\hat \Delta_{S \backslash i_1,i_2,i_{k-1},i_k} \\
    &\quad - \big[ \hat \Delta_{i_1,i_2} - \hat \Delta_{i_1} \hat \Delta_{i_2}\big]\big[\hat \Delta_{S\backslash i_1,i_2} - \hat \Delta_{i_k} \hat \Delta_{S\backslash i_1,i_2,i_k} \big] \\
    &\quad - \big[ \hat \Delta_{i_{k-1},i_k} - \hat \Delta_{i_{k-1}} \hat \Delta_{i_k}\big]\big[\hat \Delta_{S\backslash i_{k-1},i_k} - \hat \Delta_{i_1} \hat \Delta_{S\backslash i_1,i_{k-1},i_k} \big]\\
    &\quad +\big[ \hat \Delta_{i_1,i_2} - \hat \Delta_{i_1} \hat \Delta_{i_2}\big]\big[ \hat \Delta_{i_{k-1},i_k} - \hat \Delta_{i_{k-1}} \hat \Delta_{i_k}\big]\hat \Delta_{S \backslash i_1,i_2,i_{k-1},i_k},
\end{align*}
and note that the quantity 
$$\text{sgn}(K'_{i_1,i_{k-1}} K'_{i_{k-1},i_k} K'_{i_k,i_2} K'_{i_2,i_1})$$
is, by construction, computable using the signs of three- and four-cycles in our basis. Based on Lemma \ref{lm:z_form}, we set
\begin{align*}
    &\text{sgn}\bigg( 2 \,  (-1)^{k+1}  \, K'_{i_k,i_1} \prod_{j=1}^{k-1} K'_{i_j,i_{j+1}} \prod_{(a,b) \in U} \left[1-\frac{\epsilon_{i_a,i_{a+1}} K'_{i_{b-1},i_a} K'_{i_{a+1},i_b}}{K'_{i_a,i_{a+1}} K'_{i_{b-1},i_b}}\right] \bigg) = \text{sgn}(\hat Z),
 \end{align*}
with the set $U$ defined as in Lemma \ref{lm:z_form}. From this equation, we can compute $s_{K'}(C)$, as the quantity
$$\text{sgn}\left[1-\frac{\epsilon_{i_a,i_{a+1}} K'_{i_{b-1},i_a} K'_{i_{a+1},i_b}}{K'_{i_a,i_{a+1}} K'_{i_{b-1},i_b}}\right]$$
is computable using the signs of three- and four-cycles in our basis (see Section \ref{sec:identify} for details). In the case of $|C|>4$, we note that $s_{K'}(C)$ was defined using $O(1)$ principal minors, all corresponding to subsets of $V(C)$, and previously computed information regarding three- and four-cycles (which requires no additional querying). \\
\end{adjustwidth}

\noindent {\bf Step 3:} Define $K'$. \\

\begin{adjustwidth}{1.5em}{0pt}
We have already computed $K'_{i,i}$, $i \in [N]$, $|K'_{i,j}|$, $i \ne j$, and $s_{K'}(C)$ for every cycle in a simple cycle basis. The procedure for producing this matrix is identical to Step 5 of the \textsc{RecoverK} algorithm. \\
\end{adjustwidth}

\subsection{A Proof of Theorem \ref{thm:apx_alg}}

In this subsection, we prove the following theorem.

\begin{theorem}[Restatement of Theorem \ref{thm:apx_alg}]
Let $K \in \mathcal{K}_N(\alpha,\beta,\gamma)$ and $\{\hat \Delta_S\}_{S \subseteq [N]}$ satisfy $| \hat \Delta_S - \Delta_S(K) | \le \delta $ for all $S \subseteq [N]$, for some 
$$\delta < \alpha \min\{ \alpha^{3 \phi_G}, \beta^{3 \phi_G/2},\gamma \}/100,$$
where $G$ is the sparsity graph of $(\Delta_S)_{|S|\le 2}$. The \textsc{RecoverK$^\approx$}$\big(\{\hat \Delta_S\}, \, \delta \big)$ algorithm outputs a matrix $K'$ satisfying 
$$\rho(K,K') \le 3 \delta/2 \alpha.$$
This algorithm runs in time polynomial in $N$, and queries at most $O(N^2)$ principal minors, all of order at most $3 \, \phi_G$.
\end{theorem}

The proof of this theorem adds little insight and consists primarily of verifying that $\delta$ is small enough so that $K$ and $K'$ have the same charged sparsity graph and $s(\cdot)$ agrees for $K$ and $K'$ on all cycles in our basis. The bound that we make repeated use of is that, because $|\Delta_S|,|\hat \Delta_S|\le 1$, $S \subseteq [N]$, and $0<\delta<1$,
\begin{equation}\label{eqn:error}
\big| \Delta_{S_1} ... \Delta_{S_k} - \hat \Delta_{S_1} ... \hat \Delta_{S_k} \big| \le 1 - (1 - \delta)^k \le k \delta
\end{equation}
for any $S_1,..., S_k \subseteq [N]$, where $k \le 5$. We begin by computing element-wise errors. For diagonal entries $K'_{i,i}$,
$$|K'_{i,i} - K_{i,i}| = | \hat \Delta_i - \Delta_i| \le \delta \quad \text{for all } i \in [N].$$
For off-diagonal entries $K'_{i,j}$, $i \ne j$, we note that 
$$ |K_{i,j} K_{j,i}| \ge \alpha^2 > 6 \delta \quad \text{for all } i \ne j,$$
and so, by Step 1 of the algorithm (see Inequality (\ref{eqn:offdiag})), $\text{sgn}(K'_{i,j} K'_{j,i}) = \text{sgn}(K_{i,j} K_{j,i})$ for all $i \ne j$, and $G_K = G_{K'}$. In addition, by Inequality (\ref{eqn:offdiag}),
\begin{equation}\label{eqn:error_offdiag}
    \big| |K_{i,j}| - |K'_{i,j}| \big| = \frac{\big| |K_{i,j}K_{j,i}| - |K'_{i,j} K'_{j,i}| \big|}{ |K_{i,j}| + |K'_{i,j}|} \le \frac{3\delta}{2\alpha} \quad \text{for all } i \ne j.
\end{equation}
Using Inequality (\ref{eqn:error_offdiag}), we can produce a version of Inequality (\ref{eqn:error}) for products of off-diagonal entries, i.e., because $|K_{i,j}|,| K'_{i,j}|\le \sqrt{2}$ and $\big| |K_{i,j}| - |K'_{i,j}| \big| \le 3\delta/2\alpha < \sqrt{2}$,
\begin{equation}\label{eqn:offdiag_est}
\big| | K_{i_1,j_1} ... K_{i_k,j_k}| - |K'_{i_1,j_1} ... K'_{i_k,j_k}| \big| \le 2^{k/2}\bigg( 1 - \bigg[1 - \frac{3 \delta}{2 \sqrt{2} \alpha}\bigg]^k \bigg) \le 2^{\tfrac{k-3}{2}}  \, \frac{ 3 \, k \,  \delta}{\alpha}
\end{equation}
for any $i_1 \ne j_1,i_2 \ne j_2,...,i_k \ne j_k$, where $k\le 5$. To complete the proof, it suffices to verify that $s_K(C)$ and $s_{K'}(C)$ agree for all cycles in our basis. Similar to Step 2 of our algorithm, we will investigate cycles of length three, then four, and then of length greater than four. \\

\noindent {\bf Case I:} $|C|=3$. \\

Let $C = i \, j \, k \, i$, $i<j<k$. Using Condition \ref{cond:approx} and Inequality (\ref{eqn:error}), we have $|K_{i,j}K_{j,k}K_{k,i}| \ge \alpha^3$ and
\begin{align*}
    &\big|K_{i,j}K_{j,k}K_{k,i} -\big[ \hat \Delta_i \hat \Delta_j \hat \Delta_k - \big[ \hat \Delta_i \hat \Delta_{j,k} + \hat \Delta_j \hat \Delta_{i,k}  + \hat \Delta_k \hat \Delta_{i,j}  \big]/2 +  \hat \Delta_{i,j,k}/2 \big] \big|\\
    &\qquad\qquad \le \big| \Delta_i \Delta_j \Delta_k - \hat \Delta_i \hat \Delta_j \hat \Delta_k \big| + \big|  \Delta_i  \Delta_{j,k} - \hat \Delta_i \hat \Delta_{j,k} \big|/2 + \big|  \Delta_j \Delta_{i,k} - \hat \Delta_j \hat \Delta_{i,k} \big|/2 \\
    &\qquad\qquad \quad + \big|  \Delta_k  \Delta_{i,j} - \hat \Delta_k \hat \Delta_{i,j} \big|/2 + \big| \Delta_{i,j,k} - \hat \Delta_{i,j,k} \big|/2 \\
    &\qquad\qquad \le 3 \delta + 3 \times(2 \delta)/2 + \delta/2 < 7 \delta < \alpha^3,
\end{align*}
and so $s_{K'}(C) = s_{K}(C)$. This completes the case $|C| = 3$.\\

\noindent {\bf Case II:} $|C|=4$. \\

Now, let $C = i \, j \, k \, \ell \, i$, $i<j<k<\ell$. Using Condition \ref{cond:approx} and Inequality (\ref{eqn:error}), we have $Z \ge 2 \gamma$ ($Z$ defined in Equation (\ref{eqn:z_4cycle})) and 
\begin{align*}
    \big| Z - \hat Z \big| &\le \big| \Delta_{i,j,k,\ell} - \hat \Delta_{i,j,k,\ell} \big| + 
     \big| \Delta_{i,j} \Delta_{k,\ell} - \hat \Delta_{i,j} \hat \Delta_{k,\ell} \big| + 
      \big|\Delta_{i,k}  \Delta_{j,\ell} - \hat \Delta_{i,k} \hat \Delta_{j,\ell} \big| \\ &\quad + 
       \big|\Delta_{i,\ell} \Delta_{j,k}  - \hat \Delta_{i,\ell} \hat \Delta_{j,k}  \big| + 
        \big|\Delta_i \Delta_{j,k,\ell} - \hat \Delta_i \hat \Delta_{j,k,\ell} \big| + 
         \big| \Delta_j \Delta_{i,k,\ell}  - \hat \Delta_j \hat \Delta_{i,k,\ell}  \big| \\ &\quad + 
          \big|\Delta_k \Delta_{i,j,\ell} - \hat \Delta_k \hat \Delta_{i,j,\ell} \big| + 
           \big|\Delta_\ell \Delta_{i,j,k} - \hat \Delta_\ell \hat \Delta_{i,j,k} \big| + 2
            \big|  \Delta_i \Delta_j \Delta_{k,\ell} - \hat \Delta_i \hat \Delta_j \hat \Delta_{k,\ell}\big| \\ &\quad + 2 
             \big|\Delta_i \Delta_k \Delta_{j,\ell} - \hat \Delta_i \hat \Delta_k \hat \Delta_{j,\ell} \big| + 2
              \big| \Delta_i  \Delta_\ell \Delta_{j,k} - \hat \Delta_i \hat \Delta_\ell \hat \Delta_{j,k}\big| \\ &\quad + 2
               \big| \Delta_j \Delta_k \Delta_{i,\ell} - \hat \Delta_j \hat \Delta_k \hat \Delta_{i,\ell} \big|  + 2
                \big| \Delta_j \Delta_\ell \Delta_{i,k} - \hat \Delta_j \hat \Delta_\ell \hat \Delta_{i,k}\big|\\ &\quad + 2
                 \big| \Delta_k \Delta_\ell \Delta_{i,j} - \hat \Delta_k \hat \Delta_\ell \hat \Delta_{i,j} \big| + 6
                  \big|  \Delta_i  \Delta_j \Delta_k \Delta_\ell  - \hat \Delta_i \hat \Delta_j \hat \Delta_k \hat \Delta_\ell \big| \\
                  &\le \delta + 7 \times   2\delta  + 6 \times [ 2 \times 3 \delta] + 6 \times 4 \delta = 75 \delta < 2 \gamma,
\end{align*}
and so $\text{sgn}(Z) = \text{sgn}(\hat Z)$. If $C$ is the only positive four cycle of $G[\{i,j,k,\ell\}]$, then $s_{K'}(C) = s_{K}(C)$. If $G[\{i,j,k,\ell\}]$ has more than one positive four cycle, then there exists a unique choice of $\phi_1,\phi_2,\phi_3$ such that Equation (\ref{eqn:z4_phi}) holds, and any other choice $\phi_1',\phi_2',\phi_3'$ satisfies
\begin{align*}
    &\big| \phi_1' |K_{i,j} K_{j,k} K_{k,\ell} K_{\ell,i}| + \phi_2' |K_{i,j} K_{j,\ell} K_{\ell,k} K_{k,i}| + \phi_3' |K_{i,k} K_{k,j} K_{j,\ell} K_{\ell,i}| + Z/2 \big| \\
    &\qquad \qquad \ge \min \{ 2 \alpha^4, \, 2 \alpha^2 \beta, \, 2 \gamma\},
\end{align*}
where the lower bound $2 \alpha^4$ applies when $(\phi_1',\phi_2',\phi_3')$ and $(\phi_1,\phi_2,\phi_3)$ disagree in one variable, $ 2 \alpha^2 \beta$ applies when they disagree in two, and $2 \gamma$ applies when they disagree in three (i.e., are negations of each other). We also must bound the error due to approximately computed entries $|\hat K_{i,j}|$. Using Inequality (\ref{eqn:offdiag_est}), we have
\begin{align*}
   &\bigg| \phi_1 |K_{i,j} K_{j,k} K_{k,\ell} K_{\ell,i}| + \phi_2 |K_{i,j} K_{j,\ell} K_{\ell,k} K_{k,i}| + \phi_3 |K_{i,k} K_{k,j} K_{j,\ell} K_{\ell,i}| \\
   &\qquad - \big[ \phi_1 | K'_{i,j} K'_{j,k} K'_{k,\ell} K'_{\ell,i}| + \phi_2 |K'_{i,j} K'_{j,\ell} K'_{\ell,k} K'_{k,i}| + \phi_3 |K'_{i,k} K'_{k,j} K'_{j,\ell} K'_{\ell,i}|\big] \bigg| \le \frac{36 \, \sqrt{2} \, \delta}{\alpha}
\end{align*}
for any $\phi_1,\phi_2,\phi_3 \in \{-1,+1\}$. We now have two upper bounds, one for $|Z - \hat Z|$ and one for the error in sums of four cycles between $K$ and $\hat K$, and a lower bound for the gap between $Z$ and any choice of $(\phi_1,\phi_2,\phi_3)$ not equal to $Z$. At this point, we can now conclude that $s_{K'}(C) = s_{K}(C)$, as the lower bound for the gap is at least twice as large as the sum of the two upper bounds (properly normalized), i.e.,
$$|Z/2 - \hat Z/2| + \frac{36 \, \sqrt{2} \, \delta}{\alpha} \le \frac{75 \, \delta}{2} +  \frac{36 \, \sqrt{2} \, \delta}{\alpha} < \min \{ \alpha^4, \, \alpha^2 \beta, \, \gamma\}.$$
Here we have used the fact that $\phi_G \ge 3$ in this case.\\

\noindent {\bf Case III:} $|C|>4$. \\

Finally, let $C = i_1\, ... \, i_k \, i_1$, $k>4$, with vertices ordered to match the ordering of Lemma \ref{lm:minor_exp}, and define $S = \{i_1,...,i_k\}$. We note that we have already shown that $K$ and $K'$ agree in sign for three- and four-cycles. Using Lemma \ref{lm:z_form} and Condition \ref{cond:approx},
$$|Z| = 2 |K_{i_k,i_1}| \prod_{j=1}^{k-1} |K_{i_j,i_{j+1}}| \prod_{(a,b) \in U} \bigg|\left[1-\frac{\epsilon_{i_a,i_{a+1}} K_{i_{b-1},i_a} K_{i_{a+1},i_b}}{K_{i_a,i_{a+1}} K_{i_{b-1},i_b}}\right]\bigg| \ge 2 \alpha^{k-2|U|} \beta^{|U|}.$$

In addition, we can bound the difference between $Z$ and $\hat Z$ using a number of lengthy inequalities. In particular, we have
\begin{align*}
    \big|\Delta_S  - \hat \Delta_S \big|&\le  \delta, \\
    \big| \Delta_{i_1}  \Delta_{S\backslash i_1} -  \hat \Delta_{i_1} \hat \Delta_{S\backslash i_1}  \big|&\le 2 \delta, \\
    \big| \Delta_{i_k}\Delta_{S\backslash i_k}  -  \hat \Delta_{i_k}\hat \Delta_{S\backslash i_k}  \big|&\le 2 \delta,
\end{align*}
\begin{align*}
    &\big| \big[ \Delta_{i_1,i_k} - 2 \Delta_{i_1} \Delta_{i_k}\big]\Delta_{S\backslash i_1,i_k} - \big[ \hat \Delta_{i_1,i_k} - 2\hat \Delta_{i_1} \hat \Delta_{i_k}\big]\hat \Delta_{S\backslash i_1,i_k}  \big| \\ &\qquad \le \big|\Delta_{i_1,i_k} \Delta_{S\backslash i_1,i_k} -\hat \Delta_{i_1,i_k} \hat \Delta_{S\backslash i_1,i_k} \big| +
    2  \big|  \Delta_{i_1} \Delta_{i_k} \Delta_{S\backslash i_1,i_k} - \hat \Delta_{i_1} \hat \Delta_{i_k} \hat \Delta_{S\backslash i_1,i_k} \big| \le 8 \delta,
\end{align*}
\begin{align*}
    &2 \big|K_{i_1,i_{k-1}} K_{i_{k-1},i_k} K_{i_k,i_2} K_{i_2,i_1} \Delta_{S \backslash i_1,i_2,i_{k-1},i_k} - K'_{i_1,i_{k-1}} K'_{i_{k-1},i_k} K'_{i_k,i_2} K'_{i_2,i_1} \hat \Delta_{S \backslash i_1,i_2,i_{k-1},i_k} \big|\\
    &\qquad \le 2 \big|K_{i_1,i_{k-1}} K_{i_{k-1},i_k} K_{i_k,i_2} K_{i_2,i_1}  - K'_{i_1,i_{k-1}} K'_{i_{k-1},i_k} K'_{i_k,i_2} K'_{i_2,i_1} \big| \big| \Delta_{S \backslash i_1,i_2,i_{k-1},i_k} \big| \\ &\qquad \quad + 
    2 \big|K'_{i_1,i_{k-1}} K'_{i_{k-1},i_k} K'_{i_k,i_2} K'_{i_2,i_1}\big| \big| \Delta_{S \backslash i_1,i_2,i_{k-1},i_k} - \hat \Delta_{S \backslash i_1,i_2,i_{k-1},i_k} \big| \\
    &\qquad \le 24 \sqrt{2} \delta/\alpha +8 \delta,
\end{align*}
\begin{align*}
    &\big|\big[  \Delta_{i_1,i_{k-1}} -  \Delta_{i_1} \Delta_{i_{k-1}}\big]\big[ \Delta_{i_2,i_k} -  \Delta_{i_2}  \Delta_{i_k}\big] \Delta_{S \backslash i_1,i_2,i_{k-1},i_k} \\
    &\quad - \big[ \hat \Delta_{i_1,i_{k-1}} - \hat \Delta_{i_1} \hat \Delta_{i_{k-1}}\big]\big[ \hat \Delta_{i_2,i_k} - \hat \Delta_{i_2} \hat \Delta_{i_k}\big]\hat \Delta_{S \backslash i_1,i_2,i_{k-1},i_k} \big| \\
    &\qquad \le \big| \Delta_{i_1,i_{k-1}} \Delta_{i_2,i_k} \Delta_{S \backslash i_1,i_2,i_{k-1},i_k} - \hat\Delta_{i_1,i_{k-1}} \hat \Delta_{i_2,i_k} \hat \Delta_{S \backslash i_1,i_2,i_{k-1},i_k}\big| \\
    &\qquad \quad + \big| \Delta_{i_1,i_{k-1}}  \Delta_{i_2}  \Delta_{i_k} \Delta_{S \backslash i_1,i_2,i_{k-1},i_k} - \hat \Delta_{i_1,i_{k-1}}  \hat \Delta_{i_2}  \hat \Delta_{i_k} \hat \Delta_{S \backslash i_1,i_2,i_{k-1},i_k}\big|  \\
    &\qquad \quad + \big|\Delta_{i_1} \Delta_{i_{k-1}} \Delta_{i_2,i_k} \Delta_{S \backslash i_1,i_2,i_{k-1},i_k} - \hat \Delta_{i_1} \hat \Delta_{i_{k-1}} \hat \Delta_{i_2,i_k} \hat \Delta_{S \backslash i_1,i_2,i_{k-1},i_k}\big|  \\
    &\qquad \quad + \big|\Delta_{i_1} \Delta_{i_{k-1}}\Delta_{i_2}  \Delta_{i_k}\Delta_{S \backslash i_1,i_2,i_{k-1},i_k} - \hat \Delta_{i_1} \hat \Delta_{i_{k-1}} \hat \Delta_{i_2}  \hat \Delta_{i_k} \hat \Delta_{S \backslash i_1,i_2,i_{k-1},i_k}\big| \\
    &\qquad \le 16 \delta,
\end{align*}

\begin{align*}
    &\big|\big[ \Delta_{i_1,i_2} - \Delta_{i_1}  \Delta_{i_2}\big]\big[ \Delta_{S\backslash i_1,i_2} -  \Delta_{i_k} \Delta_{S\backslash i_1,i_2,i_k} \big] - \big[ \hat \Delta_{i_1,i_2} - \hat \Delta_{i_1} \hat \Delta_{i_2}\big]\big[\hat \Delta_{S\backslash i_1,i_2} - \hat \Delta_{i_k} \hat \Delta_{S\backslash i_1,i_2,i_k} \big] \big|\\
    &\qquad \le \big|\Delta_{i_1,i_2} \Delta_{S\backslash i_1,i_2} - \hat \Delta_{i_1,i_2} \hat \Delta_{S\backslash i_1,i_2} \big| +
    \big|\Delta_{i_1,i_2}  \Delta_{i_k} \Delta_{S\backslash i_1,i_2,i_k} - \hat \Delta_{i_1,i_2} \hat \Delta_{i_k} \hat \Delta_{S\backslash i_1,i_2,i_k}\big| \\ &\qquad \quad + 
    \big| \Delta_{i_1}  \Delta_{i_2} \Delta_{S\backslash i_1,i_2}  - \hat \Delta_{i_1} \hat \Delta_{i_2}\hat \Delta_{S\backslash i_1,i_2} \big| + 
    \big| \Delta_{i_1} \Delta_{i_2} \Delta_{i_k} \Delta_{S\backslash i_1,i_2,i_k} - \hat \Delta_{i_1} \hat \Delta_{i_2}\hat \Delta_{i_k} \hat \Delta_{S\backslash i_1,i_2,i_k}  \big| \\
    &\qquad \le 12 \delta,
\end{align*}
 and, similarly,
 
 \begin{align*}
     &\big| \big[  \Delta_{i_{k-1},i_k} -  \Delta_{i_{k-1}} \Delta_{i_k}\big]\big[ \Delta_{S\backslash i_{k-1},i_k} -  \Delta_{i_1}  \Delta_{S\backslash i_1,i_{k-1},i_k} \big]\\
     &\quad - \big[ \hat \Delta_{i_{k-1},i_k} - \hat \Delta_{i_{k-1}} \hat \Delta_{i_k}\big]\big[\hat \Delta_{S\backslash i_{k-1},i_k} - \hat \Delta_{i_1} \hat \Delta_{S\backslash i_1,i_{k-1},i_k} \big] \big| \le 12 \delta,
 \end{align*}
 and
\begin{align*}
    &\big| \big[  \Delta_{i_1,i_2} - \Delta_{i_1}  \Delta_{i_2}\big]\big[  \Delta_{i_{k-1},i_k} -  \Delta_{i_{k-1}}  \Delta_{i_k}\big] \Delta_{S \backslash i_1,i_2,i_{k-1},i_k} \\
    &\quad - \big[ \hat \Delta_{i_1,i_2} - \hat \Delta_{i_1} \hat \Delta_{i_2}\big]\big[ \hat \Delta_{i_{k-1},i_k} - \hat \Delta_{i_{k-1}} \hat \Delta_{i_k}\big]\hat \Delta_{S \backslash i_1,i_2,i_{k-1},i_k} \big|\le 16 \delta.
\end{align*}
Summing all of the above estimates, we have
$$ \big| Z - \hat Z \big| \le 77 \delta +24 \sqrt{2} \delta/\alpha< 2 \alpha^{k-2|U|} \beta^{|U|},$$
and so $Z$ and $\hat Z$ agree in sign. In addition, by Condition \ref{cond:approx},
$$\big|K_{i_a,i_{a+1}} K_{i_{b-1},i_b} -\epsilon_{i_a,i_{a+1}} K_{i_{b-1},i_a} K_{i_{a+1},i_b} \big| \ge \beta,$$
and, by Inequality (\ref{eqn:offdiag_est}),
$$\big|\big[K_{i_a,i_{a+1}} K_{i_{b-1},i_b} -\epsilon_{i_a,i_{a+1}} K_{i_{b-1},i_a} K_{i_{a+1},i_b} \big]-  \big[ K'_{i_a,i_{a+1}} K'_{i_{b-1},i_b} -\epsilon_{i_a,i_{a+1}} K'_{i_{b-1},i_a} K'_{i_{a+1},i_b}\big] \big|  \le 6 \sqrt{2}\delta/\alpha$$
for all $(a,b) \in U$, and so
$$\text{sgn}\big(K_{i_a,i_{a+1}} K_{i_{b-1},i_b} -\epsilon_{i_a,i_{a+1}} K_{i_{b-1},i_a} K_{i_{a+1},i_b}\big) = \text{sgn} \big( K'_{i_a,i_{a+1}} K'_{i_{b-1},i_b} -\epsilon_{i_a,i_{a+1}} K'_{i_{b-1},i_a} K'_{i_{a+1},i_b}\big).$$
Therefore, $s_{K'}(C) = s_{K}(C)$. This completes the analysis of the final case, and implies that $\rho(K,K')$ is given by $\max_{i,j} \big| |K_{i,j}| - |K'_{i,j}| \big| \le 3\delta/(2\alpha)$.
\end{document}